\documentclass[11pt,a4paper,dvipsnames]{article}
\usepackage{ifdraft} 

\newcommand{\COMMENT}[1]{}
\newcommand{\LONGCOMMENT}[1]{}

\usepackage{makecell,changepage,booktabs,etoc,floatpag}
\floatpagestyle{plain}
\usepackage[singlelinecheck=false]{caption}

\usepackage[utf8]{inputenc}
\usepackage[T1]{fontenc}
\usepackage{textcomp}
\usepackage{lmodern}
\usepackage{microtype}
\usepackage{amsmath,amsfonts,amssymb,amsthm,amscd,paralist,changepage,url}
\usepackage{setspace,mathrsfs}
\usepackage{etoolbox,tikz}

\usepackage{tikz-cd,dynkin-diagrams}

\usepackage[shortlabels]{enumitem}
\usepackage{mathtools}
\usepackage{xcolor,colortbl}
\usepackage[toc]{appendix}

\usepackage{array}   
\newcolumntype{L}{>{$}l<{$}} 
\newcolumntype{C}{>{$}c<{$}} 
\newcolumntype{R}{>{$}r<{$}} 

\usepackage[hidelinks,pdftex]{hyperref}
\usepackage[nameinlink,capitalize]{cleveref}
\usepackage{nameref}

\makeatletter
\let\orgdescriptionlabel\descriptionlabel
\renewcommand*{\descriptionlabel}[1]{%
  \let\orglabel\label
  \let\label\@gobble
  \phantomsection
  \edef\@currentlabel{#1\unskip}%
  \let\label\orglabel
  \orgdescriptionlabel{#1}%
}
\makeatother
\crefname{section}{\S}{\S\S}
\newtheorem{theorem}{Theorem}[section]
\newtheorem{maintheorem}{Theorem}\crefname{maintheorem}{Theorem}{Theorems}
\newtheorem{lem}[theorem]{Lemma}\crefname{lem}{Lemma}{Lemmas}
\newtheorem{cor}[theorem]{Corollary}\crefname{cor}{Corollary}{Corollaries}
\newtheorem{prop}[theorem]{Proposition}\crefname{prop}{Proposition}{Propositions}
\newtheorem*{theorem*}{Theorem}
\newtheorem*{question*}{Question}
\newtheorem*{observation*}{Observation}

\theoremstyle{definition}
\newtheorem{defn}[theorem]{Definition}

\theoremstyle{remark}
\newtheorem*{rem}{Remark}
\newtheorem{remark}[theorem]{Remark}
\newtheorem{example}[theorem]{Example}

\allowdisplaybreaks
\author{Tobias Hemmert and Marcus Zibrowius}
\title{The Witt rings of many flag varieties\\are exterior algebras}

\begin{document}
\maketitle

\begin{abstract}
  The Witt ring of a complex flag variety describes the interesting -- i.e.\ torsion -- part of its topological KO-theory.  We show that for a large class of flag varieties, these Witt rings are exterior algebras, and that the degrees of the generators can be determined by Dynkin diagram combinatorics.   Besides full flag varieties, projective spaces, and other varieties whose Witt rings were previously known, this class contains many flag varieties of exceptional types.  Complete results are obtained for flag varieties of types \(G_2\) and \(F_4\).  The results also extend to flag varieties over other algebraically closed fields.
\end{abstract}

\newcommand{\Witt}{\mathrm{W}}
\newcommand{\K}{\mathrm{K}}
\newcommand{\KO}{\mathrm{KO}}
\newcommand{\R}{\mathrm{R}}
\newcommand{\RO}{\mathrm{RO}}
\newcommand{\RSp}{\mathrm{RSp}}
\newcommand{\reducedR}{\smash{\widetilde{\R}}}
\newcommand{\reducedRO}{\smash{\widetilde{\RO}}}
\newcommand{\reducedRSp}{\smash{\widetilde{\RSp}}}
\newcommand{\h}{\mathrm{h}}

\newcommand{\ZZ}{\mathbb{Z}}
\newcommand{\ZZII}{{\ZZ/2}}
\newcommand{\NN}{\mathbb{N}}
\newcommand{\RR}{\mathbb{R}}
\newcommand{\CC}{\mathbb{C}}
\newcommand{\HH}{\mathbb{H}}
\newcommand{\rank}{\mathrm{rank}}
\newcommand{\type}{\mathrm{type}}

\newcommand{\ideal}{\mathfrak}

\newcommand{\cardinality}[1]{\left|#1\right|}
\newcommand{\Mat}{\mathrm{Mat}}
\renewcommand{\vec}[1]{\mathbf{#1}}
\renewcommand{\vector}[1]{\begin{pmatrix}#1\end{pmatrix}}
\renewcommand{\bar}[1]{\overline{#1}}
\newcommand{\FF}{\mathbb{F}}
\newcommand{\solution}{\mathcal L}
\newcommand{\Hom}{\mathrm{Hom}}
\newcommand{\im}{\mathrm{im}}
\newcommand{\contHom}{\mathrm{Hom}_{\mathrm{cont}}} 

\newcommand{\id}{\mathrm{id}}

\newcommand{\U}{\mathrm{U}}
\newcommand{\SU}{\mathrm{SU}}
\renewcommand{\O}{\mathrm{O}}
\newcommand{\SO}{\mathrm{SO}}
\newcommand{\Sp}{\mathrm{Sp}}

\newcommand{\alphaU}{\alpha_{\U}}
\newcommand{\alphaO}{\alpha_{\O}}
\newcommand{\alphaSp}{\alpha_{\Sp}}

\newcommand{\ibidem}[1]{[\textit{ibid.}\ifstrempty{#1}{}{, #1}]}

  \newcommand{\explanation}[2]{}
\newcommand{\levisubgroup}[1]{\ifstrempty{#1}{L}{L_{#1}}}
\explanation{\subgroup{}}{some centralizer of a torus in \(G\)}
\explanation{\subgroup{H}}{the centralizer of a torus in \(G\) corresponding to the subset \(H\subset\simpleroots{}\)}
\newcommand{\idcomponent}[1]{{#1}^0}
\explanation{\idcomponent{G}}{connected component of \(G\) containing the identity}
\newcommand{\roots}[1]{\mathcal R_{#1}}
\newcommand{\coroots}[1]{\mathcal R^\vee_{#1}}
\newcommand{\simpleroots}[1]{\ifstrempty{#1}{\Sigma}{#1}}
\explanation{\simpleroots}{fundamental system = set of simple roots, or the corresponding Dynkin diagram}
\newcommand{\folded}[1]{#1^\sigma}
\explanation{\folded{\simpleroots{}}}{The root system obtained from \(\simpleroots\) by folding along the automorphism \(\dual{}\)}
\newcommand{\Weyl}[1]{\mathcal{W}_{#1}}
\explanation{\Weyl{}}{Weyl group, generated by reflections in the simple roots}
\newcommand{\Weylcosets}[1]{{\Weyl{}^{#1}}}
\explanation{\Weylcosets{H}}{canonical minimal representatives of the cosets \(\Weyl{}/\Weyl{H} = \{ w\Weyl{H}\}\), as discussed in \cite[\S\,5-1]{kane}}
\newcommand{\longest}[1]{{w_o^{#1}}}
\explanation{\longest{H}}{longest element of the Weyl group \(\Weyl{H}\)}
\newcommand{\Weylo}{{\Weyl{}}_o}
\explanation{\Weylo}{subgroup of elements of Weyl group commuting with \(\longest{}\)}
\newcommand{\Weylfolded}{{\Weyl{}}^\sigma}
\explanation{\Weylfolded}{``folded'' subgroup of the Wely group; the subgroup fixed under conjugation with \(\dual{}\)}
\newcommand{\Weylfoldedparabolic}[1]{\Weylfolded_{#1}}
\explanation{\Weylfoldedparabolic{I}}{parabolic subgroup of \(\Weylfolded\) corresponding to \(I\subset \simpleroots{}\)}
\newcommand{\standardinvolution}[1]{{c_{#1}}}
\explanation{\standardinvolution{I}}{standard involution associated with \(I\subset \simpleroots{}\)}
\newcommand{\dual}[1]{{[\mathrm{\simpleroots{#1}}]}}
\explanation{\dual{I}}{involution \(-\longest{I}\)}
\newcommand{\lessdominant}[1]{<_{#1}}
\newcommand{\smallerorbit}[1]{\subset_{#1}}
\explanation{\lessdominant{}}{partial order on the dominant weights of \(G\)}
\explanation{\lessdominant{H}}{partial order on the dominant weights of \(\levisubgroup{H}\)}
\newcommand{\cellZZ}[1]{\mathcal C_{#1}}
\newcommand{\cellRR}[1]{\mathcal C_{\RR,#1}}
\newcommand{\cellclosureZZ}[1]{\overline{\mathcal C}_{#1}}
\newcommand{\cellclosureRR}[1]{\overline{\mathcal C}_{\RR,#1}}
\explanation{\cellRR{I}}{cell corresponding to subset \(I\subset \Sigma\), as defined in \cite{kane}}
\explanation{\cellclosureRR{I}}{closure of \(\cell{I}\)}
\newcommand{\closedWeylchamberZZ}[1]{\overline{\mathcal C}\ifstrempty{#1}{}{(#1)}}
\newcommand{\closedWeylchamberRR}[1]{\overline{\mathcal C}_\RR\ifstrempty{#1}{}{(#1)}}
\explanation{\closedWeylchamberRR{H}}{closed fundamental dual Weyl chamber for \(H\)\\Note that \(\closedFDWCHRR \neq\cellclosureRR{\simpleRootsH}\); in general \(\cell{\simpleRootsH}\subset\closedFDWC{}\subset\closedFDWC{H}\)}
\newcommand{\coweightspaceZZ}{{X_*}}
\newcommand{\coweightspaceRR}{{X_*^\RR}}
\newcommand{\weightspaceZZ}{{X^*}}
\newcommand{\weightspaceRR}{{X^*_\RR}}
\newcommand{\pairing}[2]{\langle #1, #2 \rangle}
\newcommand{\innerproduct}[2]{(#1,#2)}
\newcommand{\plusdim}{\ell^+}
\explanation{\plusdim(\tau)}{the dimension of the \(+1\)-Eigenspace of \(\tau\), for any involution \(\tau\in\mathrm{O}(\weightspaceRR, \invariantform{-}{-})\)}
\newcommand{\minusdim}{\ell^-}
\explanation{\minusdim(\tau)}{the dimension of \(-1\)-Eigenspace of \(\tau\), for any involution \(\tau\in\mathrm{O}(\weightspaceRR, \invariantform{-}{-})\)}
\newcommand{\fixrank}[1]{\cardinality{\simpleroots{#1}/\dual{#1}}}

\newcommand{\fixmonoid}{\closedWeylchamberZZ{H}^\dual{H}}
\explanation{\fixmonoid}{submonoid of \(\closedWeylchamberZZ{H}\) fixed by \(\dual{H}\)}
\newcommand{\symsum}[1]{{\mathcal S_{#1}}}
\newcommand{\rsymsum}[1]{{\tilde{\mathcal S}_{#1}}}
\explanation{\symsum{H}(\omega)}{\(:= \sum_{\tau\in \Weyl{H}\omega} e^{2\pi i\tau}\in R(T)^\Weyl{H}\), the elementary symmetric sum of \(\omega\) with respect to \(\Weyl{H}\)}
\explanation{\rsymsum{H}(\omega)}{\(:= \symsum{H}(\omega)-\rank(\symsum{H}(\omega))\), the reduced elementary symmetric sum of \(\omega\) with respect to \(\Weyl{H}\)}

\newcommand{\fundamentalRepresentation}[1]{{\rho_{#1}}}
\newcommand{\rfundamentalRepresentation}[1]{{\tilde\rho_{#1}}}
\newcommand{\irreducibleRepresentation}[1]{\rho_{#1}}
\explanation{\fundamentalRepresentation{\alpha}}{ fundamental representation of \(G\) corresponding to \(\alpha\)}
\explanation{\rfundamentalRepresentation{\alpha}}{\(:= \fundamentalRepresentation{\alpha}-\rank(\fundamentalRepresentation{\alpha})\), reduced fundamental representation}
\explanation{\irreducibleRepresentation{\omega}}{ irreducible representation of \(G\) with highest weight \(\omega\)}


The Witt ring of a complex flag variety can be approached in two ways.  Algebraic geometers might define a \emph{complex flag variety} as a projective homogeneous variety under some complex reductive group \(G_\CC\).  As such, it will have the form \(G_\CC/P\) for some parabolic subgroup \(P\).  There is a \(\ZZ/4\)-graded multiplicative cohomology theory \(\Witt^*_{\text{alg}}(-)\) on algebraic varieties, due to Balmer \cite{balmer1}, which extends the usual notion of the Witt ring of quadratic forms over a field.  The Witt rings under investigation in this paper are precisely the rings \(\Witt^*_{\text{alg}}(G_\CC/P)\).

From a topological point of view, we may equivalently define a complex flag variety as a quotient manifold \(G/\levisubgroup{}\) obtained from a compact Lie group \(G\) by dividing out a Levi subgroup \(L\), i.e.\ the centralizer of some torus.  Indeed, as a manifold, \(G_\CC/P\) is diffeomorphic to such a quotient of the maximal compact subgroup \(G\subset G_\CC\).  Moreover, the Witt ring of \(G/L\) can be defined purely topologically as the \(\ZZ/4\)-graded ring given in degree \(i\) by
\[
  \Witt^i_{\text{top}}(G/L) :=  \tfrac{\KO^{2i}(G/L)}{\K(G/L)}
\]
Here, \(\KO^*\) denotes real topological K-theory, \(\K\) denotes complex topological K-theory, and the Witt group \(\Witt^i\) is the quotient of the former by the latter under the realification map.
Fortunately, the two definitions agree \cite{zibrowius:cellular}.  So we are free to drop the subscripts and work with the topological definition in all that follows.  (But see \cref{rem:AG} for the algebraic point of view.)  As explained in \cite[\S\,1.1]{zibrowius:koff}, the Witt ring captures all torsion in \(\KO^*(G/L)\).

Our starting point is the following result of the second author \cite[Theorem~3.3]{zibrowius:koff}:
\begin{maintheorem}\label{thm:old}
  Let $G$ be a simply-connected compact Lie group, and let $T\subset G$ be a maximal torus.
  The Witt ring of the \emph{full flag variety} \(G/T\) is an exterior algebra over \(\ZZII\) on \(\cardinality{\simpleroots{}/\dual{}}\) generators (see below).
  More precisely, the Witt ring has $b_\HH$ generators of degree 1 and $\frac{b_\CC}{2}+b_\RR$ generators of degree 3, where $b_\CC$, $b_\RR$ and $b_\HH$ denote the number of fundamental representations of $G$ of complex, real and quaternionic type, respectively.
\end{maintheorem}

We need to explain the notation \(\cardinality{\simpleroots{}/\dual{}}\) for the number of generators.  Associated with the pair \((G,T)\), we have a root system and a Dynkin diagram \(\simpleroots{}\) whose nodes correspond to a set of simple roots.  Moreover, there is a canonical involution \(\dual{}\) on \(\simpleroots{}\), given by the action of the negative of the longest element of the Weyl group~\(\Weyl{}\).
We denote by the \(\simpleroots{}/\dual{}\) the set of orbits of \(\dual{}\), so \(\cardinality{\simpleroots{}/\dual{}}\) is simply the number of \(\dual{}\)-orbits.  The non-trivial orbits correspond to pairs of roots whose associated fundamental representations are mutually dual representations of complex type.  So \(\cardinality{\simpleroots{}/\dual{}} = \tfrac{b_\CC}{2}+b_\RR + b_\HH\), i.e.\ the generator counts appearing in the theorem are consistent.

\begin{remark}\label{rem:trivial-involution}
  In most cases, the involution is trivial: it is trivial whenever \(\simpleroots{}\) is a disjoint union of types \(A_1\), \(B_k\), \(D_{2k}\), \(E_7\), \(E_8\), \(F_4\) or \(G_2\) \cite[\S\,27-2]{kane}.  For types \(A_n\), \(D_{2k+1}\) and \(E_6\), the involution is as indicated in the following graphical depictions:
  \[
    \dynkin[involutions={15;24}] A{ooooo}
    \quad
    \dynkin[involutions={[in=120,out=60,relative]45}] D{ooooo}
    \quad
    \dynkin[involutions={16;35}] E{oooooo}
  \]
\end{remark}

How does \cref{thm:old} generalize to other types of flag varieties \(G/\levisubgroup{}\)?  The explicit calculations in \cite{hemmert1} for \(G\) of classical type indicate that we may expect the Witt ring of \(G/\levisubgroup{}\) to be a tensor product over \(\ZZII\) of \(\Witt^0(G/\levisubgroup{})\) and an exterior algebra on generators of odd degrees.
However, these calculations rest on a case-by-case analysis, and shed little light on the relation between, say, the number of generators and the combinatorial properties of the root systems involved.  In the current paper, we exhibit a large class of flag varieties for which \(\Witt^*(G/\levisubgroup{})\) is simply an exterior algebra, and for which the number of generators can be described in a fashion entirely analogous to that in \cref{thm:old}.

To describe both this class of varieties and the result, let us briefly recall the classification of flag varieties \(G/\levisubgroup{}\) under a fixed connected group \(G\).  We may assume that \(G\) is simply-connected, so that \(G\) is fully determined by the Dynkin diagram \(\simpleroots{}\).  Replacing \(\levisubgroup{}\) by a conjugate subgroup as necessary, we may assume that the root system of \(\levisubgroup{}\) is the closed sub-root system generated by the simple roots corresponding to some subdiagram  \(\simpleroots{H} \subset \simpleroots{}\).   In fact, there is a one-to-one correspondence between complex flag varieties under \(G\), up to \(G\)-diffeomorphism, and subdiagrams \(\simpleroots{H}\subset \simpleroots{}\), up to \(\Weyl{}\)-equivalence (see \cref{prop:classification} below).  We may thus specify a flag variety \(G/\levisubgroup{H}\) by displaying a Dynkin diagram \(\simpleroots{}\) with a subdiagram \(\simpleroots{H}\) whose nodes are marked black.  There is a simple graphical calculus for determining \(\Weyl{}\)-equivalent subsets of a Dynkin diagram, see for example \cite[\S\,28]{kane}.
\begin{example}
  The diagram \(\dynkin A{ooooo}\) specifies \(\SU(6)/T\), the classical variety of flags of vector subspaces of \(0\subsetneq V_1 \subsetneq \dots \subsetneq V_5 \subsetneq \CC^6\)
\end{example}
\begin{example}\label{eg:P5}
  The diagram \(\dynkin A{o****}\) specifies \(\SU(6)/U(5)\), the complex projective space \(\CC\mathbb{P}^5\)
\end{example}
\begin{example}\label{eg:SU6max}
  The diagrams \(\dynkin A{*oooo}\), \(\dynkin A{o*ooo}\), \(\dynkin A{oo*oo}\), \(\dynkin A{ooo*o}\), \(\dynkin A{ooooo*}\) each specify one connected component of the variety parametrizing incomplete flags \(0\subsetneq V_1 \subsetneq V_2 \subseteq V_3 \subsetneq V_4 \subsetneq \CC^6\).  These components are all isomorphic.  As manifolds, they are even \(\SU(6)\)-equivariantly isomorphic, corresponding to the fact that the marked subsets are \(\Weyl{}\)-equivalent.
\end{example}
\begin{example}[spinor variety]\label{eg:spinor}
  The diagrams \(\dynkin D{****o}\), \(\dynkin D{***o*}\) and \(\dynkin B{***o}\) specify isomorphic varieties.  More explicitly, the third diagram corresponds to the spinor variety \(\mathrm{Gr}_{\SO}(4,9)\) of \(4\)-dimensional isotropic subspaces of \(\CC^9\)  with respect to a fixed non-degenerate symmetric bilinear form. The first two diagrams correspond to the two different connected components of the variety \(\mathrm{Gr}_{\SO}(5,10)\) of \(5\)-dimensional isotropic subspaces of \(\CC^{10}\).  Both components are isomorphic to \(\mathrm{Gr}_{\SO}(4,9)\).  They are naturally homogeneous spaces under \(\SO(10)\), and they are mutually \(\SO(10)\)-equivariantly diffeomorphic, corresponding to the fact that the first two marked subsets are \(\Weyl{}\)-equivalent.
  (More generally, the two connected components of \(\mathrm{Gr}_{\SO}(n,2n)\) are \(\SO(2n)\)-equivariantly diffeomorphic for odd \(n\), but only non-equivariantly diffeomorphic for even \(n\).)
\end{example}
\begin{example}\label{eg:F4-exception}
  The diagram \(\dynkin F{o***}\) specifies the variety of subspaces of a complex Albert algebra that are \(6\)-dimensional and on which the multiplication is trivial \cite[\S\,9.1]{carrgaribaldi}.
\end{example}
Our main result applies to flag varieties for which the diagrams \(\simpleroots{}\) and \(\simpleroots{H}\) are both connected, as is the case in \cref{eg:P5,eg:SU6max,eg:spinor,eg:F4-exception} above.
Just as we have a canonical involution \(\dual{}\) on \(\simpleroots{}\), we have a canonical involution \(\dual{H}\) on \(\simpleroots{H}\).  We will show that the number of orbits of each involution determines the number of generators of the Witt ring:
\begin{maintheorem}\label{thm:main}
  Consider a complex flag variety \(G/\levisubgroup{H}\).  Assume that \(G\) and the semisimple part of \(\levisubgroup{H}\) are simple, or, equivalently, that both \(\simpleroots{}\) and \(\simpleroots{H}\) are connected.  The Witt ring of $G/\levisubgroup{H}$ is an exterior algebra on \(\fixrank{} - \fixrank{H}\) generators of odd degrees in each of the following cases:
  \begin{compactitem}
  \item \(G\) is of one of the classical types \(A_n\), \(B_n\), \(C_n\), \(D_n\), and \(\levisubgroup{H}\) is \emph{not} of type \(D_{2k}\) (for and \(k>1\)).
  \item \(G\) is of exceptional type, and the flag variety \(G/\levisubgroup{H}\) is marked with a \(\checkmark\) in \cref{table:results-connected}.
  \end{compactitem}
\end{maintheorem}

Evidently, more cases are excluded when \(G\) is of exceptional type than when \(G\) is of classical type.  However, while the claims for \(G\) of classical type are known or at least easily deducible from \cite{hemmert1},
the results for \(G\) of exceptional type are mostly new, and in this sense more interesting.  The proof we provide is new in all cases.  It is based on the observation that the root data of the flag varieties listed in \cref{thm:main} share a simple property that can be stated independently of any Witt theory: see condition~\ref{cond:singlecell} in \cref{sec:overview}.  At least for some of the flag varieties with \(\simpleroots{}\) and \(\simpleroots{H}\) connected that violate this condition, the conclusion of \cref{thm:main} is indeed not valid; see \cref{eg:DD} for the excluded classical case.

\begin{table}
  \thisfloatpagestyle{empty}
  \newlength{\offset}
  \setlength{\offset}{1ex}
  \newcommand{\extfails}[1]{
    \begin{tikzpicture}
      \node[inner sep=0] (image) at (0,0) {#1};
      \draw[gray] (image.south east) -- (image.north west);
      \draw[gray] (image.north east) -- (image.south west);
    \end{tikzpicture}
  }
  \newcommand{\extunknown}[1]{
    \begin{tikzpicture}
      \node at (0,0) {\textcolor{gray}{\Huge \textbf{?}}};
      \node[inner sep=0] at (0,0) {#1};
    \end{tikzpicture}
  }
  \newcommand{\singlecell}[1]{
    \begin{tikzpicture}
      \node[inner sep=0] (image) at (0,0) {#1};
      \node at ($(image.north east)+(0,\offset)$) {\checkmark};
    \end{tikzpicture}
  }
  \newcommand{\orbitbasis}[1]{
    \begin{tikzpicture}
      \node[inner sep=0] (image) at (0,0) {#1};
      \node at ($(image.north east)+(0,\offset)$) {(ob)};
    \end{tikzpicture}
  }
  \newcommand{\exttrue}[1]{
    \begin{tikzpicture}
      \node[inner sep=0] (image) at (0,0) {#1};
      \node at ($(image.north east)+(0,\offset)$) {(ext)};
    \end{tikzpicture}
  }
  \begin{adjustwidth}{-4cm}{-4cm}
    \centering
    \renewcommand{\arraystretch}{1.5}
    \setlength{\offset}{1ex}
    \begin{tabular}{L}
      G_2\\
      \\
      \singlecell{\dynkin G{oo}}\\
      \\
      \singlecell{\dynkin G{*o}}\\
      \singlecell{\dynkin G{o*}}\\
    \end{tabular}
    \begin{tabular}{L}
      \hspace{1em}F_4 \\
      \\
      \singlecell{\dynkin F{oooo}}\\
      \\
      \singlecell{\dynkin F{*ooo}}\\
      \singlecell{\dynkin F{ooo*}}\\
      \singlecell{\dynkin F{**oo}}\\
      \singlecell{\dynkin F{oo**}}\\
      \\
      \orbitbasis{\dynkin F{***o}}\\
      \\
      \exttrue{\dynkin F{o***}}\\
    \end{tabular}
    \setlength{\offset}{-1ex}
    \begin{tabular}{L}
      \hspace{1.5em}E_6 \\
      \\
      \singlecell{\dynkin E{oooooo}}\\
      \\
      \singlecell{\dynkin E{*ooooo}}\\
      \singlecell{\dynkin E{*o*ooo}}\\
      \singlecell{\dynkin E{*o**oo}}\\
      \singlecell{\dynkin E{*o***o}}\\
      \orbitbasis{\dynkin E{*o****}}\\
      \\
      \extunknown{\dynkin E{o****o}}\\
      \extfails{\dynkin E{*****o}}\\
    \end{tabular}
    \begin{tabular}{L}
      \hspace{2em}E_7\\
      \\
      \singlecell{\dynkin E{ooooooo}}\\
      \\
      \singlecell{\dynkin E{*oooooo}}\\
      \singlecell{\dynkin E{*o*oooo}}\\
      \singlecell{\dynkin E{*o**ooo}}\\
      \singlecell{\dynkin E{*o***oo}}\\
      \orbitbasis{\dynkin E{*o****o}}\\
      \singlecell{\dynkin E{o*o****}}\\
      \orbitbasis{\dynkin E{*o*****}}\\
      \\
      \extunknown{\dynkin E{o****oo}}\\
      \extunknown{\dynkin E{o*****o}}\\
      \extunknown{\dynkin E{o******}}\\
      \\
      \singlecell{\dynkin E{******o}}\\
    \end{tabular}
    \begin{tabular}{L}
      \hspace{2em}E_8\\
      \\
      \singlecell{\dynkin E{oooooooo}}\\
      \\
      \singlecell{\dynkin E{*ooooooo}}\\
      \singlecell{\dynkin E{*o*ooooo}}\\
      \singlecell{\dynkin E{*o**oooo}}\\
      \singlecell{\dynkin E{*o***ooo}}\\
      \extunknown{\dynkin E{*o****oo}}\\
      \extunknown{\dynkin E{*o*****o}}\\
      \extunknown{\dynkin E{*o******}}\\
      \\
      \extunknown{\dynkin E{o****ooo}}\\
      \extunknown{\dynkin E{o*****oo}}\\
      \extunknown{\dynkin E{o******o}}\\
      \extunknown{\dynkin E{o*******}}\\
      \\
      \singlecell{\dynkin E{******oo}}\\
      \extunknown{\dynkin E{*******o}}\\
    \end{tabular}
  \end{adjustwidth}
  \caption{Results for \(G\) of exceptional type.
    Displayed is a complete list of flag varieties \(G/\levisubgroup{H}\) with \(G\) simply-connected of exceptional type, classified according to \cref{prop:classification}.
    The flag varieties marked with a \(\checkmark\) are those to which \cref{thm:main,thm:main-degrees} apply -- they satisfy the condition~\ref{cond:singlecell}.  The flag varieties marked (ob) only satisfy the weaker condition~\ref{cond:orbits}.   The flag variety of type \(F_4\) marked (ext) satisfies neither of these conditions, but its Witt ring is nevertheless an exterior algebra on generators of odd degree (see \cref{sec:F4}).  The Witt ring of the flag variety \(\mathrm{EIII}\) corresponding to the crossed out diagram is known not to be an exterior algebra on odd-degree generators (see \cref{eg:EIII}).
  }
  \label{table:results-connected}
\end{table}

\begin{table}
  \begin{center}
    \renewcommand{\arraystretch}{1.5}
    \begin{tabular}{LLLLL}
      \toprule
      A_{n, n \text{ even}}      & \dynkin A{xx.xx.xx}   & \quad & D_{n, n \text{ odd}}     & \dynkin D{**.*xx}   \\
      A_{n, n \equiv 1 \mod 4}   & \dynkin A{xx.xox.xx}  &       & D_{n, n \equiv 0 \mod 4} & \dynkin D{**.***}   \\
      A_{n, n \equiv 3 \mod 4}   & \dynkin A{xx.x*x.xx}  &       & D_{n, n \equiv 2 \mod 4} & \dynkin D{**.*oo}   \\
                                 &                       &                                  &                     \\
      B_{n, n \equiv 1,2 \mod 4} & \dynkin B{**.*o}      &       & E_6                      & \dynkin E{x*x*xx}   \\
      B_{n, n \equiv 0,3 \mod 4} & \dynkin B{**.**}      &       & E_7                      & \dynkin E{*o**o*o}  \\
                                 &                       &       & E_8                      & \dynkin E{********} \\
      C_{n, n \text{ odd}}       & \dynkin C{o*o*o*..o*} &       &                          &                      \\
      C_{n, n \text{ even}}      & \dynkin C{o*o*o*..*o} &       & F_4                      & \dynkin F{****}     \\
                                 &                       &       & G_2                      & \dynkin G{**}       \\
      \bottomrule
    \end{tabular}
  \end{center}
  \caption{The types of the fundamental representations corresponding to each node of the Dynkin diagram of a simply-connected compact Lie group \(G\) (see \cite[Table~1 at the end of Chapter~VIII]{bourbaki789} or \cite[Table~3.1]{davis:types})
    \[
      \dynkin A{x} = \text{comple}\dynkin A{x}\text{ type},\quad
      \dynkin A{o} = \text{quaterni}\dynkin A{o}\text{nic type},\quad
      \dynkin A{*} = \text{real type}
    \]
  }\label{table:types}
\end{table}

Determining the precise degrees of the generators in \cref{thm:main} requires substantially more effort.  We ultimately obtain the following uniform description.

\begin{maintheorem}\label{thm:main-degrees}
  In the cases listed in \cref{thm:old,thm:main}, the degrees of the generators of the Witt ring can be determined as follows.
  There exists a unique subdiagram \(I\subset\simpleroots{}\) with the following properties:
  \begin{compactenum}[(i)]
  \item \(I\) is symmetric in the sense that \(\dual{} I = I\).
  \item The involutions \(\dual{}\) and \(\dual{I}\) agree on \(I\).
  \item The involution \(\longest{H}\) is conjugate to \(\longest{}\longest{I} = \dual{}\dual{I}\) in the Weyl group.
  \end{compactenum}
  For this subset, \(\fixrank{I} = \fixrank{} - \fixrank{H}\).  Let \(b_\RR(I)\), \(b_\HH(I)\) and \(b_\CC(I)\)  denote the numbers of fundamental representations of real, quaternionic and complex types among the fundamental representations of \(G\) that correspond to nodes in \(I\).  The Witt ring of \(G/\levisubgroup{H}\) is an exterior algebra on \(b_\HH(I)\) generators of degree~1, and \(\left(\frac{b_\CC(I)}{2}+ b_\RR(I)\right)\) generators of degree~3.
\end{maintheorem}

While the statement may look complicated, it is typically easy to apply.  We briefly illustrate this with a few examples, old and new.

\begin{example}[Full flags]
  Consider a full flag variety \(G/T\).  In this case, \(\simpleroots{H} = \emptyset\) and \(I=\simpleroots{}\), and we recover \cref{thm:old}.
\end{example}

\begin{example}[Projective spaces]
  Consider projective space \(\mathbb{C}\mathbb{P}^5\).  In this case, \((\simpleroots{},\simpleroots{H}) = \dynkin A{o****}\), and \(\fixrank{}-\fixrank{H} = 1\).  The only subsets \(I\) satisfying conditions (i) and (ii) in \cref{thm:main-degrees} are the ones marked with black nodes as follows:
  \[
    \dynkin A{ooooo} \quad \dynkin A{oo*oo} \quad \dynkin A{o***o} \quad \dynkin A{*****}
  \]
  These are clearly distinguished by the value of \(\fixrank{I}\), so the unique subset \(I\) singled out by \cref{thm:main-degrees} is the subset
  \[
    I = \dynkin A{oo*oo}
  \]
  with \(\fixrank{I} = 1\).  From the \cref{table:types} we see that the unique node of \(I\) corresponds to a quaternionic representation of \(G = \SU(6)\).  So \(\Witt^*(\mathbb{C}\mathbb{P}^5)\) is an exterior algebra on a single generator of degree~\(1\).  We leave it as an easy exercise for the reader to recover the known results for projective spaces of arbitrary dimension in this way -- \(\Witt^*(\mathbb{C}\mathbb{P}^n)\) is trivial when \(n\) is even, and it is an exterior algebra on one generator of degree equal to  \(n \mod 4\) when \(n\) is odd.
\end{example}

\begin{example}[\(E_8\)]
  Consider a flag variety \(E_8/\levisubgroup{H}\), with \(\levisubgroup{H}\) one of the subgroups listed in \cref{table:results-connected}.  All representations of \(E_8\) are of real type.  So in this case we can conclude without determining \(I\) that \(\Witt^*(E_8/\levisubgroup{H})\) is an exterior algebra on \(8 - \fixrank{H}\) generators of degree \(3\).  To the best of our knowledge, not even the additive structure of these Witt rings was known before.
\end{example}

In those cases in which applying \cref{thm:main-degrees} does require finer information on conjugacy classes of involutions, one may simply refer to the results tabulated in \cite{aux:involutions} to find the correct subsets \(I\subset \simpleroots{}\).  We mention just two examples:

\begin{example}[spinor varieties]
  Consider the spinor variety \(\mathcal S_n := \mathrm{Gr}_\SO(n-1,2n-1)\), described by the the diagram \(B_{n-1}\) with the marked subdiagram of type \(A_{n-2}\) (see \cref{eg:spinor}).  The involution \(\longest{H}\) is conjugate to the (standard) involution \(\longest{H'}\) corresponding to a subdiagram \(H'\) of type \((A_1)^{\lceil\frac{n-2}{2}\rceil}\):
  \[
    \begin{cases}
      \dynkin B{*o*o.*oo} & \text{ if \(n\) is even}\\
      \dynkin B{*o*o.o*o} & \text{ if \(n\) is odd}
    \end{cases}
  \]
  In this case, the subdiagram \(I\) determined by \cref{thm:main-degrees} is the following diagram of type \((A_1)^{\lfloor\frac{n}{2}\rfloor}\):
  \[
    \begin{cases}
      \dynkin B{*o*o.*o*} & \text{ if \(n\) is even}\\
      \dynkin B{*o*o.o*o} & \text{ if \(n\) is odd}
    \end{cases}
  \]
  (So for odd \(n\), \(I = H'\)).
  Indeed, the first two conditions of \cref{thm:main-degrees} are clearly satisfied for \(I\).  For the third condition, see \cite{aux:involutions}.
  We thus find that \(\Witt^*(\mathcal S_n)\) is an exterior algebra on \(\lfloor\frac{n}{2}\rfloor\) generators.  When \(n\equiv 2,3 \mod 4\), all generators are of degree~\(3\).  When \(n\equiv 0,1 \mod 4\), one generator is of degree~\(1\) and all others are of degree~\(3\).  This is consistent with older additive results (see \cite[\S\,4.6]{zibrowius:cellular}),
  which have recently been extended to spinor varieties over general bases \cite{hmx:spinor}.
\end{example}
\begin{example}[\(\mathrm{EVII}\)]
  The hermitian symmetric space \(\mathrm{EVII}\) is the homogeneous space under \(E_7\) corresponding to \(\dynkin E{******o}\).
  The involution \(\longest{H}\) is conjugate to the involution \(\longest{K}\) defined by the subset \(\dynkin E{o****o}\).
  In this case, \(\longest{K}\) is conjugate to \(\longest{}\longest{I} = -\longest{I}\) for the subset \(I = \dynkin E{o*oo*o*}\) \cite{aux:involutions}.  This subset \(I\) also satisfies the remaining conditions in \cref{thm:main-degrees}.  From \cref{table:types}, we see that the nodes of \(I\) correspond precisely to those fundamental representations of \(E_7\) which are quaternionic.  So \(\Witt^*(\mathrm{EIII})\) is an exterior algebra on three generators of degree one.  This is again compatible with the known additive results, see \cite{konohara:hermitian} and \cite[\S\,4.7]{zibrowius:cellular}.
\end{example}

Our final two examples indicate that some of the restrictions in \cref{thm:main,thm:main-degrees} are necessary:
\begin{example}[\(D_n \supset D_{2k}\)]\label{eg:DD}
  The flag varieties of classical type described by \(D_n \supset D_{2k}\) are explicitly excluded in \cref{thm:main}. Theorem~4.13 in \cite{hemmert1} shows that \cref{thm:main} indeed does not extend to these flag varieties.  (While their Witt rings are also exterior algebras, they have two more generators than \cref{thm:main} would suggest.)
\end{example}

\begin{example}[\(\mathrm{EIII}\)]\label{eg:EIII}
  The hermitian symmetric space \(\mathrm{EIII}\) is the homogeneous space under \(E_6\) corresponding to \(\dynkin E{*****o}\).
  As indicated in \cref{table:results-connected}, \cref{thm:main} does not apply in this case.  Indeed, \(\Witt^*(\mathrm{EIII})\) is known not to be an exterior algebra on generators of odd degrees. (The ring is concentrated in degree zero but non-trivial; see \cite{konohara:hermitian} and \cite[\S\,4.7]{zibrowius:cellular}.)
\end{example}

We therefore believe that the restrictions in \cref{thm:main} are close to optimal, i.e.\ that for most flag varieties not covered by the Theorem, the Witt ring is more complicated.  However, surprisingly, the Witt rings of flag varieties \(F_4/\levisubgroup{H}\) are always exterior algebras on the expected number of generators -- even when \(\levisubgroup{H}\) has several simple factors.  In the final section of this paper, we will show:

\begin{maintheorem}\label{thm:F4}
  The Witt ring of a flag variety $F_4/\levisubgroup{H}$ is an exterior algebra on $4-\fixrank{H}$ generators of degree~\(3\).
\end{maintheorem}

\begin{remark}\label{rem:AG}
  In algebraic geometry, we can also consider flag varieties over fields \(F\) other than \(\CC\).  The results of this paper remain true for flag varieties \(G/P\) over any algebraically closed field of characteristic not two.  Here, \(G\) is the simply-connected simple algebraic group corresponding to \(\simpleroots{}\), and \(P\) is the parabolic subgroup determined by \(\simpleroots{H}\) and a choice of Borel subgroup.  Indeed, by \cite[Corollary~4.1]{xie:sos}, the Witt groups of these flag varieties do not depend on the base field.  Alternatively, the proofs below, which are formulated mostly in the language of root data, can be performed directly in the more general setting.
  In the remainder of the paper, we will nonetheless stick to the topological point of view, because we feel that the results are accessible to a wider audience in this way. The results are new even in this setting.
\end{remark}

\subsection*{Applications to the Stable Converse Soul Question}
A Riemannian manifold \(M\) has the \emph{stable converse soul property} if, for every real vector bundle \(E\) over \(M\), the “stabilized” vector bundle \(E\times \RR^k\) admits a metric of non-negative sectional curvature, for sufficiently large \(k\).  \Cref{thm:old} was used in \cite{SCSQ} to show that a few low-rank full flag varieties indeed have this property \ibidem{Theorem~1.6}.  The same paper also exhibited two larger families of simply-connected compact homogeneous spaces that have this property:  all such spaces of dimension at most seven \ibidem{Theorem~1.7}, and a family containing those flag varieties \(G/\levisubgroup{H}\) for which the involution \(\dual{H}\) is trivial \ibidem{Theorem~1.8}.  \Cref{thm:main,thm:main-degrees} above can be used to produce further examples, in view of the following observation:
\begin{observation*}
  If \(\Witt^*(G/\levisubgroup{H})\) is an exterior algebra on at most three generators of degree~1, or an exterior algebra on at most three generators of degree~3, then \(G/\levisubgroup{H}\) has the stable converse soul property.
\end{observation*}
Indeed, this is immediate from the proof of \ibidem{Proposition~5.6}.  As examples, consider the flag varieties \(X_1\) and \(X_2\) specified by the following two diagrams:
\[
  \dynkin E{*o*ooo} \quad\quad \dynkin F{oo**}
\]
In either case, \cref{thm:main,thm:main-degrees} imply that \(\Witt^*(X_i)\) is an exterior algebra on 3 generators of degree~3, and hence that \(X_i\) has the stable converse soul property.  Note that these examples fall into none of the classes mentioned above:  they are not full flags, they have large dimensions (66 and 42, respectively), and the involution \(\dual{H}\) is non-trivial.
\subsection*{Acknowledgments}
This paper generalizes and extends results obtained by the first author in his thesis \cite{hemmert:thesis}.  In particular, \cref{thm:F4} was already included in \cite{hemmert:thesis} as Theorem~7.13.  Also, the special case of \cref{thm:main} in which \(\simpleroots{H}\) is of type \(A_1\) was already present as Theorem~7.9, up to a gap filled by \cref{sec:simplify-setup} below.  Some preliminary results for \(G=E_7\), superseded by the results presented here, were obtained by Kamran Alizadeh Rad \cite{kar:thesis}. \Cref{prop:real-iff-odd}, or some version thereof, was communicated to the second author by Ian Grojnowski a long time ago. Skip Garibaldi supplied some references and many constructive comments.

Many computations, preliminary and final, were vastly simplified by the Macaulay2 package \texttt{WeylGroups} \cite{M2:WeylGroups} written by Baptiste Calmès and Viktor Petrov.  The Dynkin diagrams were typeset using the \texttt{dynkin-diagrams} package developed by Ben McKay.

\newpage

\tableofcontents

\vspace{2cm}
\section{Overview}
\label{sec:overview}
The computation of the Witt ring \(\Witt^*(G/\levisubgroup{H})\) can be reduced to computations in the character lattice \(\weightspaceZZ\) of a maximal torus \(T\subset \levisubgroup{H} \subset G\), and to an analysis of the actions on \(\weightspaceZZ\) of the Weyl groups \(\Weyl{}\) and \(\Weyl{H}\) of \((G,T)\) and \((G,\levisubgroup{H})\), respectively.   In short, this reduction proceeds as follows.
The Witt ring is determined by the complex K-ring \(\K(G/\levisubgroup{H})\) together with the usual involution on \(\K(G/\levisubgroup{H})\) that sends a vector bundle to its dual.  This K-ring in turn can be computed from the complex representation rings \(\R G\) and \(\R(\levisubgroup{H})\), and these representations rings can be identified with the subrings of  \(\R T\) that are invariant under the Weyl groups \(\Weyl{}\) and \(\Weyl{H}\), respectively.  Here, \(\R T \cong \ZZ[\weightspaceZZ]\), and the actions of the Weyl groups on \(\R T\) are induced from their actions on \(\weightspaceZZ\).   Moreover, under all these identifications, the usual involution on \(\K(G/\levisubgroup{H})\) is induced by the canonical involution \(\dual{H}\) on \(\weightspaceZZ\).

The simple roots in \(\simpleroots{H}\) determine a Weyl chamber \(\closedWeylchamberZZ{H}\subset \weightspaceZZ\), which is preserved by the involution \(\dual{H}\).  We will eventually reduce all claims to an analysis of the cone of fixed points \(\closedWeylchamberZZ{H}^\dual{H}\) of this action.  \Cref{thm:main} rests on a chain of implications between the following conditions:

\renewcommand{\implies}{}
\begin{description}
\item[(connected$^+$)\label{cond:connected}]\mbox{}\\
  \((G,\levisubgroup{H})\) are as permitted in \cref{thm:main}.
\implies
\item[(single cell)\label{cond:singlecell}]~with parameter \(I\subset \simpleroots{}\)\\
  The cone of \(\dual{H}\)-fixed points \(\closedWeylchamberZZ{H}^\dual{H}\) is a translate under the action of \(\Weyl{}\) of the \(\dual{}\)-fixed points of a face \(\cellclosureZZ{I}\) of the fundamental Weyl chamber \(\closedWeylchamberZZ{}\) of \(G\).  In symbols:  \(\closedWeylchamberZZ{H}^\dual{H} = w(\cellclosureZZ{I}^\dual{})\) for some \(w\in\Weyl{}\).
  \implies
\item[(orbit basis)\label{cond:orbits}]~with parameter \(I\subset\simpleroots{}\)\\
  Consider the fundamental weights \(\omega_{\alpha}\in\weightspaceZZ\) of \(G\) corresponding to the simple roots \(\alpha\in\simpleroots{}\). For a \(\dual{}\)-orbit \([\alpha]\subset\simpleroots{}\), define
  \begin{align*}
    \omega_{[\alpha]} := \begin{cases}
      \omega_{\alpha} &\text{ if } \dual{}\alpha = \alpha\\
      \omega_{\alpha} + \omega_{\dual{}\alpha} &\text{ if } \dual{}\alpha \neq \alpha
      \end{cases}
  \end{align*}
  Suppose \(I\subset\simpleroots{}\) is a subset that is symmetric in the sense that \(\dual{} I = I\).  Condition~\ref{cond:orbits} holds with respect to \(I\) if the following is true:
  For each \(\dual{}\)-orbit \([\alpha]\subset\simpleroots{}\setminus I\), the Weyl orbit \(\Weyl.\omega_{[\alpha]}\) intersects \(\closedWeylchamberZZ{H}^\dual{H}\) in a single weight \(\tau_{[\alpha]}\).
  Moreover, these weights \(\tau_{[\alpha]}\) form a basis of \(\closedWeylchamberZZ{H}^\dual{H}\).
  \implies
\item[(ext)\label{cond:ext}]\mbox{}\\
  The Witt ring of \(G/\levisubgroup{H}\) is an exterior algebra on \(\fixrank{}-\fixrank{H}\) generators.
\end{description}

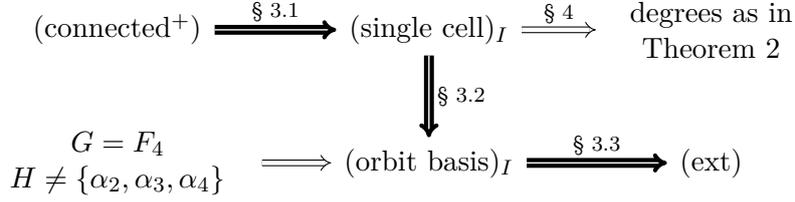
\begin{figure}
  \newcommand{\ctext}[2]{\text{\parbox{#1}{\centering #2}}}
  \[
    \begin{tikzcd}
      \text{\ref{cond:connected}} \arrow[r, ultra thick, Rightarrow, "\text{\cref{sec:connected-singlecell}}"] & \text{\ref{cond:singlecell}\(_I\)} \arrow[d, ultra thick, Rightarrow, "\text{\cref{sec:singlecell-orbits}}"] \arrow[r, Rightarrow, "\text{\cref{sec:main-degrees}}"] & \ctext{7em}{degrees as in \cref{thm:main-degrees}} \\
      \ctext{9em}{\(G = F_4\)\\\(\simpleroots{H}\neq \{\alpha_2,\alpha_3,\alpha_4\}\)} \arrow[r, Rightarrow]       & \text{\ref{cond:orbits}\(_I\)} \arrow[r, ultra thick, Rightarrow, "\text{\cref{sec:orbits-ext}}"]                     & \text{\ref{cond:ext}}
    \end{tikzcd}
  \]
  \caption{Structure of the proofs}
  \label{fig:layout}
\end{figure}

The implications between these conditions that constitute a proof of \cref{thm:main} are indicated by the bold arrows in \cref{fig:layout}.  We will begin, in \cref{sec:setup,sec:simplify-setup}, by recalling and extending the relevant general results from \cite{zibrowius:koff} and \cite{hemmert1} that establish the link between root system combinatorics and Witt rings.
The first implication is then verified in \cref{sec:connected-singlecell}.  For \(\fixrank{H} = 1\), there is a simple general argument, but in general we do need to make some case-by-case arguments for the different types of \(G\).  For \(G\) of exceptional type, this slightly tedious analysis can be cut short by an implementation in Macaulay2 \cite{M2}, using the Weyl groups package \cite{M2:WeylGroups}. The code is available at \cite{code}.  The second and third implications can be treated in a much more streamlined fashion, and this is carried out in \cref{sec:singlecell-orbits,sec:orbits-ext}, as indicated.  \Cref{thm:main} clearly follows from these implications.  While it would also be possible, and in fact easier, to obtain this result directly from \ref{cond:singlecell}, we have included a small detour via \ref{cond:orbits} to prepare for the proof of \cref{thm:F4}.

\Cref{thm:main-degrees} requires the stronger assumption~\ref{cond:singlecell}.  We need to chase the degrees of generators through all steps of the proof of \cref{thm:main}.  This is carried out in \cref{sec:main-degrees}.

Finally, the cases of \Cref{thm:F4} (for \(G\) of type \(F_4\)) that are not covered by \cref{thm:main} are proved in \cref{sec:F4}.  While condition~\ref{cond:singlecell} fails, it turns out that all but one of the remaining cases at least satisfy the slightly weaker condition \ref{cond:orbits}.  The last remaining case, with \(\levisubgroup{H}\) of type \(C_3\), requires a short explicit computation.

\section{The setup}

\subsection{Classification of flag varieties}
Consider a compact connected semisimple Lie group \(G\).  As mentioned in the introduction, a homogeneous space \(G/\levisubgroup{}\) is a projective complex variety, and hence a flag variety in the algebro-geometric sense, if and only if \(\levisubgroup{}\) is the centralizer of a torus  in \(G\) \cite[Théorèmes 1, 2]{serre1}. These subgroups \(\levisubgroup{}\) are occasionally referred to as the \emph{Levi subgroups} of \(G\) (e.g.\ \cite[after Theorem~11]{vogan:challenges}).  They can be classified as follows:

\begin{prop}\label{prop:classification}
  Fix a compact connected Lie group \(G\) and a maximal torus \(T\subset G\). Let \(\Weyl{}\) denote the associated Weyl group, and \(\simpleroots{}\) an associated set of simple roots.  For a subset \(I\subset \simpleroots{}\), let \(\levisubgroup{I}\) denote the centralizer of \(\idcomponent{\left(\cap_{\alpha\in I} \ker \alpha\right)}\).  This defines a bijective correspondence:
  \begin{align*}
    \frac{\{ \text{subsets of } \simpleroots{} \}}{\Weyl{}\text{-equivalence}}
    &\xrightarrow{\quad 1:1\quad }
    \frac{\{ \text{ centralizers of tori in \(G\)} \}}{\text{ conjugation in \(G\)}}
    \\
    I \quad & \quad\mapsto\quad\quad \levisubgroup{I}
  \end{align*}
\end{prop}
Note also that we may always pass to a simply-connected cover of \(G\).  We then find that any complex flag variety has the form \(G/\levisubgroup{I}\) for some semisimple simply-connected compact Lie group \(G\), and some \(I\subset \simpleroots{}\).   Moreover, \(G/\levisubgroup{I}\) and \(G/\levisubgroup{J}\) are \(G\)-equivariantly diffeomorphic (as manifolds) if and only if the subsets \(I,J\subset \simpleroots{}\) are \(\Weyl{}\)-equivalent.

\begin{proof}%
  \newcommand{\rootss}[1]{\roots{}(#1)}
  The explicit definition of \(\levisubgroup{I}\) shows that \(\levisubgroup{I}\) and \(\levisubgroup{J}\) are conjugate whenever the subsets \(I\) and \(J\) are \(\Weyl{}\)-equivalent.  It is verified in \cite[Prop.~1.8]{hemmert:thesis} that the resulting map \(\phi\colon I\mapsto \levisubgroup{I}\) is surjective.  Moreover, \(\phi\) is split injective:  Recall that, by a theorem of Hopf, any centralizer \(\levisubgroup{}\) of a torus in \(G\) has maximal rank.  Thus, up to conjugation, we may assume that \(T\subset \levisubgroup{}\).  Then the root system \(\rootss{\levisubgroup{},T}\) is a closed sub-root system of the root system \(\rootss{G,T}\). (Here, we already use some notation from \cref{sec:notation} below.)   As any torus is monogenic, we may identify \(\levisubgroup{}\) with the centralizer of a one-parameter subgroup \(\gamma\colon \RR \to T\).  Viewing \(\gamma\) as an element of vector space of coweights \(\coweightspaceRR\), we may describe the root system of \((\levisubgroup{},T)\) explicitly as follows:
  \[
    \rootss{\levisubgroup{},T} = \{ \alpha \in \rootss{G,T} \mid \pairing{\gamma}{\alpha} = 0 \}
  \]
  It follows from this description that \(\rootss{\levisubgroup{},T}\) is the span of a subset \(I'\) of some fundamental system \(\simpleroots{}'\) of the root system of \((G,T)\).  As the Weyl group \(\Weyl{}\) acts transitively on fundamental systems, there is some \(w\in\Weyl{}\) such that \(w.\simpleroots{}' = \simpleroots{}\).  Define \(\Psi(\levisubgroup{}) := w.I' \subset \simpleroots{}\).   This defines a map \(\Psi\) in the opposite direction, such that \(\Psi\circ\Phi = \id\) \cite[Remark~1.10]{hemmert:thesis}.
\end{proof}

\subsection{Notation}\label{sec:notation}
We make the standing assumption that \(G\) is a simply-connected connected compact Lie group, that \(T\subset G\) is a maximal torus and that \(T\subset\levisubgroup{H}\subset G\), with \(\levisubgroup{H}\) as specified in \cref{prop:classification}.  The following  notation will be used throughout the paper; some of it has been introduced already.  We mostly follow \cite[\S\,2]{vogan:challenges}.
We write
\begin{align*}
  \weightspaceZZ &:= \{\text{ continuous group homomorphisms } T \to S^1 \;\}\\
  \coweightspaceZZ &:= \{\text{ continuous group homomorphisms } S^1 \to T\;\}\\
  \pairing{-}{-}&\colon \coweightspaceZZ\times\weightspaceZZ \to \ZZ
\end{align*}
for the lattice of characters\slash weights of \(T\), for the lattice of coweights, and for the canonical pairing under which these two lattices are mutually dual.  We denote by
\(\Weyl{}\) the Weyl group of \((G,T)\), by \(\roots{}\subset\weightspaceZZ\) the set of roots and by \(\coroots{}\subset\coweightspaceZZ\) the set of coroots of \((G,T)\).  The corresponding data for \((\levisubgroup{H},T)\) is decorated with a subscript, so that altogether we have root data
\((\weightspaceZZ,\roots{},\coweightspaceZZ,\coroots{})\) and
\((\weightspaceZZ,\roots{H},\coweightspaceZZ,\coroots{H})\)
associated with \((G,T)\) and \((\levisubgroup{H},T)\), respectively.  As is customary, we write \(\alpha\mapsto \alpha^\vee\) for the bijection \(\roots{}\cong\coroots{}\) that is implicit in this data.

Next, recall that we have fixed sets of simple roots
\begin{align*}
  \simpleroots{}&\subset \roots{}  &   \simpleroots{H} &= \simpleroots{} \cap \roots{H}
\end{align*}
for \((G,T)\) and \((\levisubgroup{H},T)\).  They determine closed Weyl chambers
\begin{align*}
  \closedWeylchamberZZ{} &:= \{ \tau\in\weightspaceZZ \mid \pairing{\alpha^\vee}{\tau} \geq 0 \text{ for all } \alpha\in\simpleroots{}\}\\
  \closedWeylchamberZZ{H} &:= \{ \tau\in\weightspaceZZ \mid \pairing{\alpha^\vee}{\tau} \geq 0 \text{ for all } \alpha\in\simpleroots{H}\}
\end{align*}
The weights in \(\closedWeylchamberZZ{}\) are the dominant weights with respect to our choice of simple roots \(\simpleroots{}\).  As \(G\) is simply-connected, they form a free abelian monoid generated by certain fundamental weights \(\omega_\alpha\in\weightspaceZZ\) indexed by the simple roots \(\alpha\in\simpleroots{}\).  These fundamental weights are uniquely defined by the requirement that \(\pairing{\alpha^\vee}{\omega_\alpha} = 1\) and \(\pairing{\alpha^\vee}{\omega_\beta} = 0\) for all \(\alpha \neq \beta\) in \(\simpleroots{}\).  There are partial orders \(\lessdominant{}\) and \(\lessdominant{H}\) on \(\closedWeylchamberZZ{}\) and \(\closedWeylchamberZZ{H}\) that can be described either in terms of convex hulls of Weyl orbits or in terms of the coefficients of weights when written as (\(\RR\)-)linear combinations of simple roots \cite[Chapter\,6, Proposition~2.4]{btd}.

Our choices of simple roots \(\simpleroots{}\) and \(\simpleroots{H}\) also determine distinguished sets of generators of the Weyl groups \(\Weyl{}\) and \(\Weyl{H}\) as Coxeter groups.  We write \(\longest{}\in\Weyl{}\) and \(\longest{H}\in\Weyl{H}\) for the longest elements with respect to these generators.  Viewing the Weyl groups as groups of automorphisms of \(\weightspaceZZ\), we define the involutions
\begin{align*}
  \dual{} &:= -\longest{}  & \dual{H} &:= -\longest{H}
\end{align*}
The involution \(\dual{}\) permutes the simple roots \(\simpleroots{}\).  The involution \(\dual{H}\) permutes the simple roots in \(\simpleroots{H}\), and acts as \(-1\) on the complement
\(
  \{\tau \in\weightspaceZZ \mid \pairing{\alpha^\vee}{\tau} = 0 \text{ for all } \alpha \in H\}
\).
As already pointed out in \cref{rem:trivial-involution}, for most types of \(G\), the \(\dual{} = \id\). In general, however, \(\dual{}\) is not even an element of the Weyl group. Indeed, the following three conditions are equivalent \cite[\S\,27-2]{kane}:
\begin{equation*}
  \dual{} = \id, \quad
  -\id \in \Weyl{}, \quad
  \dual{} \in \Weyl{}
\end{equation*}
We will occasionally pass from lattices to vector spaces, and from submonoids to real cones:
\begin{align*}
  \weightspaceRR &:= \weightspaceZZ\otimes_\ZZ\RR\\
  \closedWeylchamberRR{} &:= \closedWeylchamberZZ{}\otimes_\ZZ\RR\\
                 &\;= \{ \tau \in \weightspaceRR \mid \pairing{\alpha^\vee}{\tau} \geq 0 \text{ for all } \alpha\in\simpleroots{}\}\\
  \closedWeylchamberRR{H} &:= \closedWeylchamberZZ{H}\otimes_\ZZ\RR
\end{align*}
The vector space \(\weightspaceRR\) can be identified with the dual tangent space of \(T\). Under this identification, \(\weightspaceZZ\) corresponds to the integral lattice, denoted \(I^*\) in classical sources such as \cite{btd} or \cite{adams}.

The roots \(\roots{}\subset \weightspaceRR\) are the usual root system associated with \((G,T)\)
with respect to a \(\Weyl{}\)-invariant inner product on \(\weightspaceRR\), which we denote \(\innerproduct{-}{-}\).
Much of the literature on compact Lie groups is formulated in terms of root systems rather than root data, and we will freely pass between these two points of view when quoting results from different sources.  The main equality to keep in mind is that
  \begin{equation}\label{eq:innerproduct-versus-pairing}
    \pairing{\alpha^\vee}{\omega} = \frac{2\innerproduct{\alpha}{\omega}}{\innerproduct{\alpha}{\alpha}}
  \end{equation}
for any root \(\alpha\in\roots{}\) and any \(\omega\in \weightspaceRR\).  The roots \(\roots{H}\subset \weightspaceRR\) form a sub-root system. (We adhere to the convention that a root system need not generate the full ambient vector space.)

Some more notation regarding subcones\slash cells in \(\weightspaceRR\) is introduced at the beginning of \cref{sec:connected-singlecell}.  Some of the notation introduced above with respect to a fixed subset \(\simpleroots{H}\subset \simpleroots{}\) is freely extended to other subsets.

\subsection{Witt rings as Tate cohomology rings}
\label{sec:setup}
There is a close link between the Witt ring of \(G/\levisubgroup{H}\) and the representation theory of \(G\) and \(\levisubgroup{H}\), summarized in \cref{eq:main-iso-1} below.  We briefly explain where it comes from.  First, a lemma of Bousfield identifies the Witt ring of \(G/\levisubgroup{H}\) with the Tate cohomology ring of the complex K-ring \(\K^0(G/\levisubgroup{H})\), equipped with the usual involution \(\circ\) that sends a vector bundle to the dual bundle:
\begin{equation}\label{eq:Bousfield}
  \Witt^*(G/\levisubgroup{H}) \xrightarrow{\;\cong\;} \h^*(\K^0(G/\levisubgroup{H}),\circ)
\end{equation}
\cite[\S\,1.2]{zibrowius:koff}.   Here, the Tate cohomology \(\h^*(R,\circ)\) of a ring with involution \((R,\circ)\) is the \(\ZZII\)-graded ring with components \(\h^\pm (R,\circ) := {\ker(\id \mp \circ)}/{\im(\id\pm \circ)}\).   The isomorphism \eqref{eq:Bousfield} is induced by the complexification of vector bundles.  It respects the grading in the sense that elements of even degrees get sent to \(\h^+(\K^0(G/\levisubgroup{H}))\), and elements of odd degrees get sent to \(\h^-(\K^0(G/\levisubgroup{H}))\).

Second, Hodgkin's Theorem describes the complex K-theory \(\K^0(G/\levisubgroup{H})\) in terms of the complex representation rings \(\R G\) and \(\R(\levisubgroup{H})\):
\begin{equation}\label{eq:Hodgkin}
  \K^0(G/\levisubgroup{H}) \xleftarrow[\;\cong\;]{\alphaU} \R(\levisubgroup{H})\otimes_{\R G}\ZZ
\end{equation}
\cite[\S\,2.3]{zibrowius:koff}\cite[Lemma~9.2]{hodgkin1}.
Here, \(\R(\levisubgroup{H})\) is viewed as an \(\R G\)-module via the restriction \(\R G\to\R(\levisubgroup{H})\), and \(\ZZ\) is viewed as an \(\R G\)-module via the rank map.  The map \(\alphaU\) sends a complex \(\levisubgroup{H}\)-representation \(V\) to the associated homogeneous vector bundle \(V\times_{\levisubgroup{H}} G\) over \(G/\levisubgroup{H}\).\footnote{
  We use the notation \(\alphaU\) instead of the notation \(\alpha\) from \cite{zibrowius:koff} to alleviate the clash with our use of \(\alpha\) to denote simple roots.
}

The representation ring \(\R G\) is a polynomial ring on generators indexed by the simple roots \(\simpleroots{}\).  There is some choice regarding the specific generators.  A common choice is to use the irreducible representations \(\fundamentalRepresentation{\alpha}\) with highest weights the fundamental weights \(\omega_\alpha\).  We may equally well use the symmetric sums \(\symsum{}(\omega_{\alpha})\) or their reduced variants \(\rsymsum{}(\omega_\alpha) := \symsum{}(\omega_\alpha)-\rank(\symsum{}(\omega_\alpha))\).
Using this notation, we can rewrite the right side of  \eqref{eq:Hodgkin} as the quotient \(\R(\levisubgroup{H})/(\rsymsum{}(\omega_\alpha)\mid\alpha\in\simpleroots{})\).  In summary, then, we have an isomorphism
\begin{equation}\label{eq:main-iso-1}
  \Witt^*(G/\levisubgroup{H}) \cong \h^*\left(\R(\levisubgroup{H})/(\rsymsum{}(\omega_\alpha) \mid \alpha \in\simpleroots{})\right).
\end{equation}
If all fundamental weights of \(G\) are self-dual, then this is precisely the isomorphism we need.  For every pair of non-self-dual fundamental weights \(\omega_\beta\), \(\omega_{\dual{}\beta}\), we could easily replace \(\rsymsum{}(\omega_\beta)\) and \(\rsymsum{}(\omega_{\dual{} \beta})\) in \eqref{eq:main-iso-1} by the single element \(\rsymsum{}(\omega_\beta)\cdot \rsymsum{}(\omega_{\dual{} \beta})\), by repeatedly applying \cite[Lemma~2.7]{hemmert:thesis}.   However, for a streamlined proof of \cref{thm:main-degrees}, it will be far more convenient to work with \(\rsymsum{}(\omega_\beta + \omega_{\dual{} \beta})\):
\begin{equation}\label{eq:main-iso-2}
  \Witt^*(G/\levisubgroup{H}) \cong \h^*\left(
    \frac{\R(\levisubgroup{H})}{
      \left(
        \substack{
          \rsymsum{}(\omega_\alpha)\\
          \rsymsum{}(\omega_\beta + \omega_{\dual{}\beta})
        }
        \;\middle\vert\;
        \substack{
          \alpha \in\simpleroots{} \text{ such that } \dual{}\alpha=\alpha\\
          \beta \in\simpleroots{} \text{ such that} \dual{}\beta\neq \beta
        }
      \right)
    }
  \right).
\end{equation}
Unfortunately, the passage from \eqref{eq:main-iso-1} to \eqref{eq:main-iso-2} is technical.  We have relegated this somewhat unsatisfying step to \cref{sec:simplify-setup}.

In any case, it is the identification \eqref{eq:main-iso-2} that our results are built on, and which motivates the analysis of root systems that follows.  As a first step towards calculating the Tate cohomology on the right side of \eqref{eq:main-iso-2}, we analyse the Tate cohomology of \(\R(\levisubgroup{H})\).  Recall that a \(\ZZ\)-basis of \(\R(\levisubgroup{H})\) is given by symmetric sums \(\symsum{H}(\omega)\) of the dominant weights \(\omega \in \closedWeylchamberRR{H}\cap\weightspaceZZ\) \cite[Theorem~6.20]{adams}, and that the product of two such symmetric sums has the form
\begin{equation}\label{eq:symsum{}product}
  \symsum{H}(\omega_1)\symsum{H}(\omega_2)
  = \symsum{H}(\omega_1+\omega_2) +  \underset{(\omega\in \closedWeylchamberRR{H}\cap \weightspaceZZ\colon \omega \lessdominant{H} \omega_1+\omega_2)}{\text{smaller terms \(\symsum{H}(\omega)\)}}
\end{equation}
\cite[Proposition~6.36]{adams}, where ``smaller terms'' refers to some linear combination of the indicated elements. Similarly, we find:
\begin{prop}\label{prop:Tate-of-H}
A \(\ZZII\)-basis of \(\h^+(\R(\levisubgroup{H}))\) is given by the classes \([\symsum{H}(\tau)]\) corresponding to the weights \(\tau\in\closedWeylchamberRR{H}^\dual{H}\cap\weightspaceZZ\). The product of two such classes has the form
\begin{equation}\label{eq:symsum{}productinTate}
  [\symsum{H}(\tau_1)][\symsum{H}(\tau_2)]
  = [\symsum{H}(\tau_1+\tau_2)] +  \underset{(\tau\in \closedWeylchamberRR{H}^\dual{H}\cap \weightspaceZZ\colon \tau \lessdominant{H} \tau_1+\tau_2)}{\text{smaller terms \([\symsum{H}(\tau)]\)}}
\end{equation}
in \(\h^+(\R(\levisubgroup{H}))\).  The negative Tate cohomology \(\h^-(\R(\levisubgroup{H}))\) vanishes.
\end{prop}
\begin{proof}
  As \(\dual{H}\symsum{H}(\omega) = \symsum{H}(\dual{H}\omega)\), the duality \(\dual{H}\) permutes the given basis of \(\R H\), fixing \(\symsum{H}(\omega)\)  if and only if it fixes \(\omega\).  This proves the additive claims.  The multiplicative formula is immediate from \eqref{eq:symsum{}product} and the observation that smaller terms \(\symsum{H}(\tau)\) with \(\dual{H}\tau\neq \tau\) necessarily occur in (dual) pairs, and hence disappear in Tate cohomology.
\end{proof}
\begin{cor}\label{lemmatatecohfreemonoid}
  If the commutative monoid \(\closedWeylchamberRR{H}^\dual{H} \cap \weightspaceZZ\) is free with basis $\tau_1,\ldots,\tau_a$, then \(\h^*(\R(\levisubgroup{H}))\) is freely generated, as a \(\ZZII\)-graded \(\ZZII\)-algebra, by the Tate cohomology classes of the symmetric sums \([\symsum{H}(\tau_i)]\in\h^+(\R(\levisubgroup{H}))\).

  More generally, under the same assumptions, \(\h^*(\R(\levisubgroup{H}))\) is freely generated by the Tate cohomology classes \([\chi_1],\dots,[\chi_a]\) of any family of \(\dual{H}\)-self-dual elements \(\chi_1,\dots,\chi_a\) in \(\R(\levisubgroup{H})\) whose classes \(\h^+(\R(\levisubgroup{H}))\) have the form
  \[
    [\chi_i] = [\symsum{H}(\tau_i)] + \underset{(\omega\in \closedWeylchamberRR{H}^\dual{H}\cap \weightspaceZZ\colon \omega \lessdominant{H} \tau_i)}{\text{ smaller terms \([\symsum{H}(\omega)]\)}}.
  \]
\end{cor}
\begin{proof}
  Consider a polynomial ring \(\ZZII[x_1,\dots,x_a]\) in \(a\) variables.
  It suffices to show that the ring homomorphism \(\varphi\colon \ZZII[x_1,\dots,x_a] \to \h^*(\R(\levisubgroup{H}))\) defined by
  \(\varphi(x_i) := [\chi_i]\) sends a basis to a basis.  In particular, it suffices to verify that the elements \(f_\tau:= \varphi(x_1^{n_1}\dots x_a^{n_a})\) with \(n_1,\dots,n_a\in\NN_0\), \(\tau := \sum_{i=1}^{a} n_i \tau_i \in \closedWeylchamberRR{H}^\dual{H}\cap\weightspaceZZ\), form a basis of \(\h^+(\R(\levisubgroup{H}))\).  Writing the basis elements of \cref{prop:Tate-of-H} as \(e_\tau := \left[\symsum{H}(\tau)\right]\), and using \eqref{eq:symsum{}productinTate}, we find that
  \(f_\tau = e_\tau + \text{ smaller terms \(e_\omega\).}\)
  As there are only finitely many dominant weights \(\omega\lessdominant{H} \tau\) for any fixed \(\tau\), this shows, by induction, that the \(f_\tau\) indeed form a basis.
\end{proof}
\Cref{prop:Tate-of-H,lemmatatecohfreemonoid} clearly indicate that we should study the monoid \(\closedWeylchamberRR{H}^\dual{H}\cap \weightspaceZZ\).

\subsection{A technical lemma}
\label{sec:simplify-setup}
As noted,  for a pair of non-self-dual fundamental weights \((\omega_\beta, \omega_{\dual{}\beta})\), we need to divide out both \(\rsymsum{}(\omega)\) and \(\rsymsum{}(\dual{}\omega)\) on the right side of \eqref{eq:main-iso-1}.  We will now show that we can divide out \(\rsymsum{}(\omega+\dual{}\omega)\) instead:

\begin{prop}
  \label{prop:simplified-quotient}
  There is canonical projection
  \begin{equation*}
    \frac{\R(\levisubgroup{H})}{
      \left(
        \substack{
          \rsymsum{}(\omega_\alpha),\\
          \rsymsum{}(\omega_\beta + \omega_{\dual{}\beta})\\
        }
        \;\middle\vert\;
        \substack{
          \alpha\in\simpleroots{} \text{ such that } \dual{}\alpha = \alpha, \\
          \beta \in\simpleroots{} \text{ such that } \dual{}\beta \neq \beta, \\
        }
      \right)}
    \twoheadrightarrow
    \frac{\R(\levisubgroup{H})}{\left(\rsymsum{}(\omega_\alpha) \;\middle\vert\; \alpha\in\simpleroots{}\right)}
  \end{equation*}
  that induces an isomorphism in Tate cohomology.
\end{prop}
All \(G\) of any irreducible type apart from \(A_n\) and \(E_6\) have at most one pair of non-\(\dual{}\)-self-dual fundamental weights.  For these types, \cref{prop:simplified-quotient} is either trivial or follows fairly directly from the following \namecref{lem:quotient-iso-on-Tate}:
\begin{lem}
  \label{lem:quotient-iso-on-Tate}
  Consider a ring with involution \((R,\circ)\).  Suppose we are given elements \(\lambda,\gamma\in R\) with the following properties:
  \begin{compactitem}
  \item \((\lambda,\lambda^\circ)\) is a regular sequence in \(R\)
  \item \(\gamma\) is a non-zero divisor in \(R\) with \(\gamma^\circ = \gamma\)
  \item \(\gamma\) lies in the ideal generated by \(\lambda\) and \(\lambda^\circ\)
  \item \(\gamma = \lambda\lambda^\circ\) in \(\h^+(R)\)
  \end{compactitem}
  Then the canonical quotient map \(R/(\gamma) \twoheadrightarrow R/(\lambda,\lambda^\circ)\) induces an isomorphism on Tate cohomology.
\end{lem}
\begin{proof}
  This can be proved in exactly the same way as \cite[Lemma~2.7]{hemmert1}.
\end{proof}
For \(G\) of type \(A_n\) or \(E_6\), we need the following additional \namecref{lem:symsum{}s-of-symmetric-pairs} to establish \cref{prop:simplified-quotient}:

\begin{lem}\label{lem:symsum{}s-of-symmetric-pairs}
  Suppose \(G\) is a simple simply-connected compact Lie group.
  For any pair of simple roots \((\alpha,\dual{}\alpha)\) with \(\alpha\neq\dual{}\alpha\), the symmetric sum \(\symsum{}(\omega_\alpha + \omega_{\dual{}\alpha})\) can be written as
  \[
    \symsum{}(\omega_\alpha + \omega_{\dual{}\alpha}) = \symsum{}(\omega_\alpha)\symsum{}(\omega_{\dual{}\alpha}) + P,
  \]
  where \(P\) is a polynomial only in the following variables: \(\symsum{}(\omega_\beta)\), where \(\beta\) ranges over those simple roots that are \(\dual{}\)-self-dual and over those non-\(\dual{}\)-self-dual roots that satisfy \(\omega_\beta + \omega_{\dual{}\beta} < \omega_\alpha + \omega_{\dual{}\alpha}\).
\end{lem}
\begin{proof}
  When there is at most one pair of non-\(\dual{}\)-self-dual simple roots, the claim is vacuous.  It therefore suffices to examine only the cases when \(\simpleroots{}\) is of type \(E_6\) or of type \(A_n\).

  For \(\simpleroots{}\) of type \(E_6\), we have \(\dual{}\alpha_1 = \alpha_6\) and \(\dual{}\alpha_3 = \alpha_5\), while \(\alpha_2\) and \(\alpha_4\) are \(\dual{}\)-self-dual in the numbering of \cite{bourbaki456}.  For the associated fundamental weights \(\omega_i\), we have \(\omega_1+\omega_6 < \omega_3 + \omega_5\).  Moreover, the only weights smaller than \(\omega_1+\omega_6\) are \(0\), \(\omega_1\) and \(\omega_6\) are the only weights smaller than \(\omega_1 + \omega_6\).  This implies the claim in this case.  (Explicitly, we have
  \(
    \symsum{}(\omega_1+\omega_6)  =  \symsum{}(\omega_1)\symsum{}(\omega_6) - 6 \symsum{}(\omega_2) - 27
  \).)

  For \(\simpleroots{}\) of type \(A_n\), we have \(\dual{}\omega_i = \omega_{n+1-i}\), and \(\omega_1+\dual{}\omega_1 < \omega_2 + \dual{}\omega_2 < \omega_3 + \dual{}\omega_3 < \dots\). In the usual coordinates, the symmetric sums \(\symsum{}(\omega)\) can be viewed as monomial symmetric polynomials.  Explicit computations with these polynomials show that, for \(i<(n+1)/2\),
  \[
    \symsum{}(\omega_i)\symsum{}(\dual{}\omega_i) = \symsum{}(\omega_i + \dual{}\omega_i) + \sum_{j\colon 0 < j < i} a_{i,j} \symsum{}(\omega_j + \dual{}\omega_j) + a_{i,0}
  \]
  for certain (non-negative) coefficients \(a_{i,j}\).  In particular, \(\symsum{}(\omega_1 + \dual{}\omega_1) = \symsum{}(\omega_1)\symsum{}(\dual{}\omega_1) - a_{1,0}\), and the claim follows by induction.
\end{proof}
\begin{rem}
  We are not aware of any type-independent proof of \cref{lem:symsum{}s-of-symmetric-pairs}.  The argument used for \(E_6\), based on the partial ordering of dominant weights, does not readily extend to \(A_n\). For example, for \(A_7\) we have \(\omega_3 < \omega_2 + \dual{}\omega_2\).  On the other hand, the argument used for \(A_n\) does not readily extend to \(E_6\): It is \emph{not} true in general that the product \(\symsum{}(\omega_\alpha)\symsum{}(\omega_{\dual{}\alpha})\) decomposes as a linear combination of symmetric sums \(\symsum{}(\omega)\) over \emph{self-dual} dominant weights \(\omega\).  For \(\simpleroots{}\) of type \(E_6\), we rather find:
  \begin{align*}
    \symsum{}(\omega_3)\symsum{}(\omega_5)
    & = \symsum{}(\omega_3+\omega_5) + 4 \symsum{}(\omega_1+\omega_2+\omega_6)\\
    &\quad + 10 \symsum{}(\omega_5+\omega_6) +  10 \symsum{}(\omega_1+\omega_3)\\
    &\quad + 15 \symsum{}(2\omega_2) + 18 \symsum{}(\omega_4) + 32 \symsum{}(\omega_1+\omega_6)\\
    &\quad + 60 \symsum{}(\omega_2) + 216
  \end{align*}
\end{rem}

\begin{proof}[Proof of \Cref{prop:simplified-quotient}]
  The idea is to repeatedly apply \cref{lem:quotient-iso-on-Tate}, but a careful proof requires a bit of notation.  Enumerate the pairs of non-\(\dual{}\)-self-dual fundamental weights as \((\omega_1,\dual{}\omega1)\), \dots \((\omega_l,\dual{}\omega_l)\) in a way that respects the order discussed in \cref{lem:symsum{}s-of-symmetric-pairs},  i.e.\ such that \(\omega_i + \dual{}\omega_i < \omega_k + \dual{}\omega_k\) only for \(i<k\).  Consider the following ideals in \(\R(G)\), and the following quotient rings:
  \begin{gather*}
    \ideal a := \left( \rsymsum{}(\omega_\alpha) \;\middle\vert\; \alpha\in\simpleroots{} \text{ such that } \dual{}\alpha = \alpha \right)\\
    \begin{aligned}
      \ideal b   & :=  \left( \rsymsum{}(\omega_i), \rsymsum{}(\dual{}\omega_i) \;\middle\vert\; k = 1, \dots l \right)&
      \ideal b^{<k} &:= \left( \rsymsum{}(\omega_i), \rsymsum{}(\dual{}\omega_i) \;\middle\vert\; 1 \leq i < k \right)\\
      \ideal d &:= \left( \rsymsum{}(\omega_i + \dual{}\omega_i) \;\middle\vert\; k = 1, \dots, l\right)&
      \ideal d^{>k} &:= \left( \rsymsum{}(\omega_i + \dual{}\omega_i) \;\middle\vert\; k < i \leq l \right)
    \end{aligned}\\
    \begin{aligned}
      (\R G)_k &:= \left.\R(G)\middle/(\ideal a + \ideal b^{<k} + \ideal d^{>k})\right.\\
      (\R(\levisubgroup{H}))_k    &:= \left.\R(\levisubgroup{H})\middle/(\ideal a + \ideal b^{<k} + \ideal d^{>k})\R(\levisubgroup{H})\right.
    \end{aligned}
  \end{gather*}
  In this notation, the rings appearing in the \namecref{prop:simplified-quotient} are:
  \begin{align*}
    \left.\R(\levisubgroup{H})\middle/(\ideal a + \ideal d)\right. &= (\R(\levisubgroup{H}))_1/(\rsymsum{}(\omega_1+\dual{}\omega_1)(\R(\levisubgroup{H}))_1\\
    \left.\R(\levisubgroup{H})\middle/(\ideal a + \ideal b)\right. &= (\R(\levisubgroup{H}))_l/(\rsymsum{}(\omega_l),\rsymsum{}(\dual{}\omega_l)(\R(\levisubgroup{H}))_l
  \end{align*}
  Moreover, for each \(k \in \{1,\dots,l-1\}\), we have the equality
  \[
    \frac{(\R(\levisubgroup{H}))_k}{\left(\rsymsum{}(\omega_k),\rsymsum{}(\dual{}\omega_k)\right)(\R(\levisubgroup{H}))_k} = \frac{(\R(\levisubgroup{H}))_{k+1}}{\rsymsum{}(\omega_{k+1} + \dual{}\omega_{k+1})(\R(\levisubgroup{H}))_{k+1}}.
  \]
  It therefore suffices to show that there are canonical projections
  \begin{equation}\label{eq:sofqbaw}
    q_k\colon
    \frac{(\R(\levisubgroup{H}))_k}{\rsymsum{}(\omega_k+\dual{}\omega_k)(\R(\levisubgroup{H}))_k}
    \twoheadrightarrow
    \frac{(\R(\levisubgroup{H}))_k}{\left(\rsymsum{}(\omega_k),\rsymsum{}(\dual{}\omega_k)\right)(\R(\levisubgroup{H}))_k}
  \end{equation}
  that induce isomorphisms on Tate cohomology.

  To this end, we first rewrite \cref{lem:symsum{}s-of-symmetric-pairs} in the following way:
  For any pair of fundamental weights \((\omega_k,\dual{}\omega_k)\), we have
  \begin{equation}\label{eq:vbifosli}
    \rsymsum{}(\omega_k + \dual{}\omega_k) = \rsymsum{}(\omega_k)\rsymsum{}(\dual{}\omega_k) + a\left(\rsymsum{}(\omega_k) + \dual{}\rsymsum{}(\omega_k)\right) + \tilde P
  \end{equation}
  in \(\R(G)\), with \(a\in\ZZ\) and \(\tilde P \in \ideal a + \ideal b^{<k}\).  This shows that  \(\ideal a + \ideal d \subset \ideal a + \ideal b\),  so we indeed have a canonical projection as stated in the \namecref{prop:simplified-quotient}. Moreover, in the quotient ring \((\R G)_k\), this simplifies to:
  \begin{equation}\label{eq:vbifosli2}
    \rsymsum{}(\omega_k+\dual{}\omega_k) = \rsymsum{}(\omega_k)\rsymsum{}(\dual{}\omega_k) + a\left(\rsymsum{}(\omega_k) + \dual{}\rsymsum{}(\omega_k)\right)
  \end{equation}
  So indeed, we have a canonical projection \(q_k\) as indicated in \eqref{eq:sofqbaw}.
  To see that \(q_k\) induces an isomorphism on Tate cohomology, we check each of the assumptions of \cref{lem:quotient-iso-on-Tate}, with \(R := (\R(\levisubgroup{H}))_k\), \(\circ := \dual{}\), \(\lambda := \rsymsum{}(\omega_k)\) and \(\gamma := \rsymsum{}(\omega_k + \dual{}\omega_k)\).
  The assumption that \(\rsymsum{}(\omega_k+\dual{}\omega_k)\) is contained in the ideal generated by \(\rsymsum{}(\omega_k)\) and \(\rsymsum{}(\dual{}\omega_k)\), and the equality of \(\rsymsum{}(\omega_k+\dual{}\omega_k)\) and  \(\rsymsum{}(\omega_k)\rsymsum{}(\dual{}\omega_k)\) in Tate cohomology, are both immediate from \eqref{eq:vbifosli2}.
  It remains to check that \((\rsymsum{}(\omega_k),\rsymsum{}(\dual{}\omega_k))\) is a regular sequence in \((\R(\levisubgroup{H}))_k\), and that \(\rsymsum{}(\omega_k+\dual{}\omega_k)\) is not a zero-divisor.
  These claims follow from following observations:
  \begin{compactenum}[(a)]
  \item \((\R(\levisubgroup{H}))_k\) is free as a module over \((\R G)_k\),
  \item \((\R G)_k\) is a free module over the polynomial ring \(S := \ZZ[\rsymsum{}(\omega_k),\rsymsum{}(\dual{}\omega_k)]\),
  \item \((\rsymsum{}(\omega_k),\rsymsum{}(\dual{}\omega_k))\) clearly is a regular sequence in \(S\), and \(\rsymsum{}(\omega_k+\dual{}\omega_k)\), expressed as on the right side of \eqref{eq:vbifosli2}, clearly is a non-zero divisor in \(S\).
  \end{compactenum}
  Observation (a) follows from the fact that \(\R(\levisubgroup{H})\) is free as a module over \(\R(G)\). For (b), write \(x_i := \rsymsum{}(\omega_i)\), and denote the duality \(\dual{}\) as \((-)^\circ\).  Then
  \[
    \R G_k \cong \frac{\ZZ[x_k,x_k^\circ,\dots,x_l,x_l^\circ]}{(f_{k+1},\dots,f_r)},
  \]
  and using \eqref{eq:vbifosli} we see that each relation \(f_i\) is of the form
  \[
   x_ix_i^\circ =  -a(x_i + x_i^\circ) + \tilde P,
  \]
  with \(\tilde P\) a polynomial in the variables \(x_j\), \(x_j^\circ\)  with \(k \leq j < i\).  So \((\R G)_k\) is a free module over \(\ZZ[x_k,x_k^\circ]\) with a basis given by all monomials
  \[
    (x_{k+1})^{n_{k+1}}(x_{k+1}^\circ)^{n'_{k+1}} \cdots (x_r)^{n_r}(x_r^\circ)^{n_r'},
  \]
  in which, for each \(i\), either \(n_i = 0\) or \(n_i' = 0\).
\end{proof}

\section{Proof of \cref{thm:main}}

\subsection{Condition~\ref{cond:connected} implies condition~\ref{cond:singlecell}}
\label{sec:connected-singlecell}
In this section, we study the geometry of the fixed-point cone \(\closedWeylchamberRR{H}^\dual{H} = \weightspaceRR^\dual{H} \cap \closedWeylchamberRR{H}\).  We begin by analysing the fixed-point space \(\weightspaceRR^\dual{H}\), before passing to the cone in \cref{prop:fixed-cone}.

As we will frequently cite results from \cite{kane}, we will work with the inner product \(\innerproduct{-}{-}\) on \(\weightspaceRR\) rather than the pairing \(\pairing{-}{-}\) in this section (see \cref{eq:innerproduct-versus-pairing} in \cref{sec:notation}.)  The root system of \((G,T)\) determines a decomposition of \(\weightspaceRR\) into disjoint \textbf{cells}.  The closed fundamental dual Weyl chamber \(\closedWeylchamberRR{}\) is a disjoint union of cells \(\cellRR{I}\) indexed by the subsets \(I\subset \simpleroots{}\), and a general cell is a translate \(w\cellRR{I}\) of such a cell for some \(w\in\Weyl{}\).  Alternatively, each cell can be defined by specifying, for each root \(\gamma\), one of the three constraints \(\innerproduct{\gamma}{-} = 0\)  or  \(\innerproduct{\gamma}{-} > 0\) or  \(\innerproduct{\gamma}{-} < 0\) \cite[\S\,5.2, Remark~2]{kane}.  Each \textbf{cell closure} \(w\cellclosureRR{I}\) can thus be defined by specifying, for each root \(\gamma\), one of the three constraints \(\innerproduct{\gamma}{-} = 0\), \(\innerproduct{\gamma}{-}\geq 0\) or \(\innerproduct{\gamma}{-} \leq 0 \).

When \(\dual{}\) is trivial (\cref{rem:trivial-involution}), the \(-1\)-eigenspace of any involution in the Weyl group is a union of cells; see \cite[Proof of proposition in \S\,27-4]{kane}.  This implies, in particular, that \(\weightspaceRR^\dual{H}\) is a union of cells in this case.  In general, we need to consider \(\dual{}\)-fixed point sets of cells instead, as the following discussion will show.

\begin{lem}\label{lem:fixed-cells}
   The following conditions on a subset \(I\subset\simpleroots{}\) and an element \(w\in\Weyl{}\) are equivalent:
  \begin{compactenum}[(a)]
  \item
    \(\dual{H} (w\cellRR{I}) = w\cellRR{I}\)
  \item
    \(\dual{} I = I\), and \(\dual{} w^{-1}\dual{H} w \in \Weyl{I}\)
  \item
    \(\dual{} I = I\), and the actions of \(w^{-1}\dual{H} w\) and \(\dual{}\) on \(\cellRR{I}\) agree
  \item
    \(\dual{} I = I\), and \((w\cellRR{I})^\dual{H} = w(\cellRR{I}^\dual{})\)
  \end{compactenum}
\end{lem}
\begin{proof}
  The equivalences are special cases of \cite[observations (1) and (2) under \S\,2.3]{steinberg:endomorphisms}, with \(w := \longest{H}\longest{}\) and \(\sigma := \dual{}\).  They can be verified directly as follows.
  For the implication (a \(\Rightarrow b)\), note that the element \(u := \dual{} w^{-1}\dual{H} w\) is certainly an element of the Weyl group \(\Weyl{}\), as the minus signs of \(\dual{}\) and \(\dual{H}\) cancel.   The assumption in (a) can be reformulated as \(u\cellRR{I} = \dual{} \cellRR{I}\).  It follows in particular that \(u\cellRR{I} \subset \closedWeylchamberRR{}\), as \(\cellRR{I} \subset \closedWeylchamberRR{}\) and as \(\dual{}\closedWeylchamberRR{} = \closedWeylchamberRR{}\).  Now, as \(\closedWeylchamberRR{}\) is a fundamental domain for the action of \(\Weyl{}\), it follows that \(u\) must fix \(\cellRR{I}\) pointwise.  Thus \(u\in \Weyl{I}\).  Moreover, it now follows that \(\cellRR{I} = \dual{}\cellRR{I}\), which is equivalent to \(I = \dual{} I\).  This proves (b).  The implication (a \(\Leftarrow\) b) is similar but simpler.
  The equivalence (b \(\Leftrightarrow\) c) and the implication (c \(\Rightarrow\) d) are immediate.
  For the final implication (c \(\Leftarrow\) d), note that (d) implies that we have two involutions \(w\dual{H} w^{-1}\) and \(\dual{}\) on \(\cellclosureRR{I}\) with identical fixed points.  These involutions are determined by their action on the free abelian monoid \(\cellclosureRR{I} \cap \weightspaceZZ\), and involutions on free abelian monoids are completely determined by their fixed points.  So in fact \(w\dual{H} w^{-1}\) and \(\dual{}\) agree on \(\cellclosureRR{I}\), implying (c).
\end{proof}
As \(\dual{H}\) permutes cells, it follows that the fixed-point space of \(\dual{H}\) can be decomposed as
\begin{equation}\label{eq:fixed-space-decomposition-1}
  \weightspaceRR^\dual{H} = \bigcup_{(I,w)} w(\cellRR{I}^\dual{}),
\end{equation}
where \((I,w)\) ranges over all pairs satisfying the equivalent conditions of \cref{lem:fixed-cells}.  Of course, different pairs may determine the same cell.  As \(\weightspaceRR^\dual{H}\) has dimension \(\plusdim(\dual{H})\), it follows in particular that \(\dim_\RR(\cellclosureRR{I}^\dual{}) \leq \plusdim(\dual{H})\) for any pair \((I,w)\) appearing here. Equivalently,
\begin{equation}\label{eq:fixed-cells}
  \cardinality{\frac{I}{\dual{}}} \geq \cardinality{\frac{\simpleroots{}}{\dual{}}} - \cardinality{\frac{\simpleroots{H}}{\dual{H}}}
\end{equation}
whenever \((I,w)\) satisfies the equivalent conditions of \cref{lem:fixed-cells}.  If we pass to cell closures in \eqref{eq:fixed-space-decomposition-1}, then it suffices to consider just those cells of maximum dimension, i.e.\ of dimension equal to \(\plusdim(\dual{H})\).
We can thus rewrite \eqref{eq:fixed-space-decomposition-1} as
\begin{equation}\label{eq:fixed-space-decomposition-2}
  \weightspaceRR^\dual{H} = \bigcup_{(I,w)} w(\cellclosureRR{I}^\dual{}),
\end{equation}
where \((I,w)\) ranges over all pairs satisfying the equivalent conditions of \cref{lem:fixed-cells}, and the additional condition that
\begin{equation}\label{eq:fixed-cells-2}
  \cardinality{\frac{I}{\dual{}}} = \cardinality{\frac{\simpleroots{}}{\dual{}}} - \cardinality{\frac{\simpleroots{H}}{\dual{H}}}
\end{equation}
We now pass to the fixed cone:
\begin{prop}\label{prop:fixed-cone}
  Consider the fixed-point cone \(\closedWeylchamberRR{H}^\dual{H} = \weightspaceRR^\dual{H} \cap \closedWeylchamberRR{H}\).
  \begin{enumerate}[(a)]
    \item
      The  set \(\closedWeylchamberRR{H}^\dual{H}\) is a convex cone of dimension \(\plusdim(\dual{H})\).  It is generated over \(\RR_{\geq 0}\) by the elements
      \(e_\theta := \omega_\theta + \dual{H}\omega_\theta\),  where \(\theta\) ranges over all \(\theta\in\simpleroots{H}\).

    \item
      The set \(\closedWeylchamberRR{H}^\dual{H}\) is a union of sets of the form \(w(\cellclosureRR{I}^\dual{})\), for certain pairs \((w,I)\) that satisfy the equivalent conditions of \cref{lem:fixed-cells} and the additional condition \eqref{eq:fixed-cells-2}.
  \item
      Suppose that no root \(\gamma\) of the root system of \((G,T)\) switches signs on \(\closedWeylchamberRR{H}^\dual{H}\), i.e.\ that every root \(\gamma\) satisfies
      \begin{equation}
        \label{eq:cell-condition}
        \tag{$\ddagger$}
        \begin{aligned}
          \text{either} &\quad \innerproduct{\gamma}{e} \geq 0 \text{ for all } e \in \closedWeylchamberRR{H}^\dual{H},\\
          \text{ or}    &\quad \innerproduct{\gamma}{e} \leq 0 \text{ for all } e \in \closedWeylchamberRR{H}^\dual{H}.
        \end{aligned}
      \end{equation}
      Then \(\closedWeylchamberRR{H}^\dual{H}\) is equal to a single set of the form \(w(\cellclosureRR{I}^\dual{})\).
      Thus, condition~\ref{cond:singlecell}\(_I\) is satisfied with respect to a parameter \(I\) satisfying \eqref{eq:fixed-cells-2}.
\end{enumerate}
\end{prop}
\begin{proof}
  (a)
  The elements \(e_\theta\) generate the fixed-point space \((\weightspaceRR)^\dual{H}\) over \(\RR\).  Discarding duplicates arising from pairs of \(\dual{H}\)-dual roots \(\{\theta,\dual{H}\theta\}\subset\simpleroots{H}\), we obtain an \(\RR\)-basis of \((\weightspaceRR)^\dual{H}\) of dimension \(\plusdim(\dual{H})\).  Each of the basis elements \(e_\theta\) lies in \(\closedWeylchamberRR{H}\), as \(\omega_\theta\) lies in \(\closedWeylchamberRR{H}\) and as \(\dual{H}\closedWeylchamberRR{H} = \closedWeylchamberRR{H}\).

  (b)
   The cone \(\closedWeylchamberRR{H}^\dual{H}\)  is obtained from \(\weightspaceRR^\dual{H}\) by intersecting with the closed Weyl chamber \(\closedWeylchamberRR{H}^\dual{H}\). This closed Weyl chamber is defined by the inequalities \(\innerproduct{\theta}{-}\geq 0\) for all \(\theta\in\simpleroots{H}\).   Given decomposition \eqref{eq:fixed-space-decomposition-2} of \(\weightspaceRR^\dual{H}\), it thus suffices to observe that  \(\innerproduct{\theta}{-}\) cannot change sign on a fixed cell \(w(\cellclosureRR{I}^\dual{})\).

   (c) The assumption implies that \(\closedWeylchamberRR{H}^\dual{H}\) is contained in a single cell closure.
   The minimal such cell closure \(w\cellclosureRR{I}\) will satisfy the equivalent conditions of \cref{lem:fixed-cells}.  Thus, we find that \(\closedWeylchamberRR{H}^\dual{H}\) is contained in a single fixed cell \((w\cellclosureRR{I})^\dual{H} = w(\cellclosureRR{I}^\dual{})\).  On the other hand, as observed above \eqref{eq:fixed-cells}, the dimension of \(w(\cellclosureRR{I}^\dual{})\) is bounded above by \(\plusdim(\dual{H})\), which by (a) is equal to the dimension of \(\closedWeylchamberRR{H}^\dual{H}\).  This implies that \(\closedWeylchamberRR{H}^\dual{H}\) is not just contained in, but equal to the fixed cell \(w(\cellclosureRR{I}^\dual{})\).
\end{proof}

It clearly suffices to evaluate the conditions \eqref{eq:cell-condition} specified in part~(c) on the generators \(e_\theta\) specified in part~(a) of \cref{prop:fixed-cone}.
When \(\plusdim(\dual{H}) = 1\), there is just one such generator, and the conditions are trivially satisfied.  Thus, condition~\ref{cond:singlecell} is always satisfied for \(\simpleroots{H}\) of Dynkin types \(A_1\) or \(A_2\).   Verifying that condition~\ref{cond:singlecell} holds more generally for all cases listed in \cref{thm:main} requires more work:

\begin{prop}
  \label{thm:small-singlecell}
  Condition~\ref{cond:singlecell} is satisfied in all cases listed in \cref{thm:main}.
\end{prop}
\begin{proof}
  As we have already remarked, it suffices to check the conditions \eqref{eq:cell-condition} in part~(c) of \cref{prop:fixed-cone} for the finitely many generators specified in part~(a) of \cref{prop:fixed-cone}.  Moreover, it suffices to check the conditions only for positive roots of \(\simpleroots{}\).

 Let us write the generators \(e_\theta\) and the positive roots \(\gamma\) as follows, for certain integer coefficients \(m_\beta^\theta\), \(a_\theta\) and \(b_\theta\):
\begin{align}
  \gamma &= \textstyle\sum_{\theta\in\simpleroots{H}} a_\theta \theta + \sum_{\beta \in \Sigma\setminus \simpleroots{H}} b_\beta \beta  \label{eq:gamma-expanded}\\
  e_\theta  &= \textstyle\omega_\theta + \omega_{\dual{H} \theta} + \sum_{\beta \in \Sigma\setminus \simpleroots{H}} m_\beta^\theta \omega_\beta \label{eq:etheta-expanded}
\end{align}
The coefficients \(m_\beta^\theta\) turn out to be non-zero only for the direct neighbours \(\beta\) of the connected of the Dynkin diagram of \(\simpleroots{H}\) that contains \(\theta\) (see  \cite[Corollary~5.5]{zibrowius:twisted}).
Up to \(\Weyl{}\)-equivalence, the subsets \(\simpleroots{H}\subset\simpleroots{}\) that we need to consider are listed in \cref{table:subsets} (classical types) and \cref{table:results-connected} (exceptional types).  Each of these subsets \(H\) is connected, has at most two neighbours \(\beta_1\) and \(\beta_2\), and for each such neighbour \(\beta\) there is a unique simple root \(\vartheta_\beta\in \simpleroots{H}\) to which it is connected in the diagram.  The coefficient \(m_\beta^\theta\) can be described more explicitly using this neighbour \(\vartheta_\beta\).  This yields the following reformulation of our condition:

\begin{table}
  \begin{adjustwidth}{-4cm}{-4cm}
    \centering
  \begin{tabular}{LLLCCCC}
    \toprule
    \simpleroots{} & \simpleroots{H}         &                                                                                    & \vartheta_{\beta} & (\bar C^{\theta_i,\vartheta_{\beta}}+\bar C^{\dual{H}\theta_i,\vartheta_{\beta}}) & C_{\vartheta_\beta,\beta} & \frac{\innerproduct{\beta}{\beta}}{\innerproduct{\theta_i}{\theta_i}}\\
    \midrule
    A_n            & A_{k\quad  2 < k < n}   & \dynkin[labels={\theta_1,\theta_2,\theta_k,\beta}]     A{**.*o.o}                  & \theta_k  & 1 & -1 & 1\\
    B_n            & A_{k\quad  2 < k < n-1} & \dynkin[labels={\theta_1,\theta_2,\theta_k,\beta}]     B{**.*o.oo}                 & \theta_k & 1 & -1 &  1 \\
    B_n            & A_{n-1}                 & \dynkin[labels={\theta_1,\theta_2,,,\theta_{n-1},\beta}] B{**.***o}                & \theta_{n-1} & 1 & -2 & \tfrac{1}{2}\\
    B_n            & B_{k\quad  1 < k < n}   & \dynkin[labels={,,\beta,\theta_1,\theta_2,\theta_k}]    B{oo.o*.**}                & \theta_1 & \begin{cases} 2 & \text{for } i \neq k \\ 1 & \text{for } i = k \end{cases} & -1 & \begin{cases} 1 & \text{for } i \neq k \\ 2 & \text{for } i = k \end{cases}\\
    C_n            & A_{k\quad  2 < k < n-1} & \dynkin[labels={\theta_1,\theta_2,\theta_k,\beta}]     C{**.*o.oo}                 & \theta_k & 1 & -1 & 1\\
    C_n            & A_{n-1}                 & \dynkin[labels={\theta_1,\theta_2,,,\theta_{n-1},\beta}] C{****.*o}                & \theta_{n-1} & 1 & -1 & 2 \\
    C_n            & C_{k\quad  1 < k < n}   & \dynkin[labels={,,\beta,\theta_1,\theta_2,\theta_k}]    C{oo.o*.**}                & \theta_1 & 2 & -1 & \begin{cases} 1 & \text{for } i \neq k \\ \tfrac{1}{2} & \text{for } i = k \end{cases} \\
    D_n            & A_{k\quad  2 < k < n-2} & \dynkin[labels={\theta_1,\theta_2,\theta_k,\beta}]     D{**.*o.ooo}                & \theta_k & 1 & -1 & 1\\
    D_n            & A_{n-2}                 & \dynkin[labels={\theta_1,\theta_2,,,\theta_{n-2},\beta_1,\beta_2}]     D{**.***oo} & \theta_{n-2} & 1 & -1 & 1\\
    D_n            & A_{n-1}                 & \dynkin[labels={\theta_1,\theta_2,,,,\theta_{n-1},\beta}]     D{**.****o}          & \theta_{n-2} & \begin{cases} 1 & \text{for } i \in \{1,n-1\} \\ 2 & \text{for } i  \in \{ 2,\dots,n-2 \} \end{cases} & -1 & 1\\
    D_n            & A_3                     & \dynkin[labels={,,,\beta,,\theta_1,\theta_3}]     D{oo.oo***}                      & \theta_2 & \begin{cases} 1 & \text{for } i \in \{1,3\} \\ 2 & \text{for } i = 2 \end{cases} & -1 & 1\\
    D_n            & D_{k\quad 3 < k < n}    & \dynkin[labels={,\beta,\theta_1,\theta_2,,\theta_{k-1},\theta_k}]     D{o.o**.***}  & \theta_1 & \begin{cases}2 & \text{for } i \in \{1,\dots,k-2\} \\  1 & \text{for } i\in\{k-1,k\} \end{cases} & -1 & 1\\
    \bottomrule
  \end{tabular}
\end{adjustwidth}
  \caption{The cases with classical \(G\) that need to be considered in the proof of \cref{thm:main}.}
  \label{table:subsets}
\end{table}

\begin{lem}\label{lem:numerical-criterion}
  Suppose (for simplicity) that \(\simpleroots{H}\) is connected.
  The conditions \eqref{eq:cell-condition} in part~(c) of \cref{prop:fixed-cone}
  are equivalent to the following numerical criterion: for every positive root \(\gamma\) of \((G,T)\), the expression
  \[
    \Delta_\gamma(\theta) := (a_\theta + a_{\dual{H} \theta}) + \sum_{\beta}  b_\beta\cdot (\bar C^{\theta,\vartheta_\beta} + \bar C^{\dual{H} \theta, \vartheta_\beta})\cdot C_{\vartheta_\beta,\beta} \cdot \tfrac{\innerproduct{\beta}{\beta}}{\innerproduct{\theta}{\theta}}
  \]
  is either non-positive for all \(\theta \in \simpleroots{H}\) or non-negative for all \(\theta \in \simpleroots{H}\).
  Here \(a_\theta\geq 0\) and \(b_\theta\geq 0\) are as in \eqref{eq:gamma-expanded}, \(C_{\alpha,\beta} = \innerproduct{\alpha}{\beta^\vee}\) is the Cartan matrix of the root system of \((G,T)\), and \(\bar C^{\sigma,\theta}\) is the inverse of the Cartan matrix of the subroot system spanned by \(\simpleroots{H}\).  The sum is over all neighbours \(\beta\) of \(\simpleroots{H}\), and \(\vartheta_\beta\) denotes the unique simple root \(\vartheta_\beta\in\simpleroots{H}\) connected to \(\beta\).
\end{lem}
\begin{proof}
  Recall that the fundamental weights \(\omega_j\) are defined by the property that \(\pairing{\alpha_i^\vee}{\omega_j} = \delta_{ij}\).  In view of \cref{eq:innerproduct-versus-pairing}, this means that \(\innerproduct{\omega_i}{\alpha_j} = \delta_{ij} \frac{\innerproduct{\alpha_i}{\alpha_i}}{2}\).  Moreover,   \(m_\beta^\theta = (\bar{C}^{\theta,\vartheta_\beta} + \bar{C}^{\dual{H} \theta,\vartheta_\beta})\cdot C_{\vartheta_\beta,\beta}\) by \cite[Corollary~5.5]{zibrowius:twisted}.
  The above expression therefore differs from \(\innerproduct{\gamma}{e_\theta}\) only by a factor \(\tfrac{\innerproduct{\theta}{\theta}}{2}\).
\end{proof}
The coefficients \(a_\theta\) and \(b_\beta\) of all positive roots \(\gamma\) are conveniently tabulated in \cite[Plates I--IX]{bourbaki456}. The factors \((\bar C^{\theta,\vartheta_\beta} + \bar C^{\dual{H} \theta,\vartheta_\beta})\) can be read off \cite[Figure~1]{zibrowius:twisted},
the factors \(C_{\vartheta_\beta,\beta}\) can be read off \cite[Figure~5]{zibrowius:twisted}, and the final factor \(\innerproduct{\beta}{\beta}/\innerproduct{\theta}{\theta}\) can also be read off the Dynkin diagram of \((G,T)\) (see \cref{table:final-factor}).
\begin{table}
    \centering
    \begin{tabular}{lL}
      \toprule
      path between \(\theta\) and \(\beta\) has \dots                       & \tfrac{\innerproduct{\beta}{\beta}}{\innerproduct{\theta}{\theta}} \\
      \midrule
      only single edges              & 1                                                        \\
      double edge towards  \(\beta\) & \tfrac{1}{2}                                             \\
      triple edge towards \(\beta\)  & \tfrac{1}{3}                                             \\
      double edge towards \(\theta\) & 2                                                        \\
      triple edge towards \(\theta\) & 3                                                        \\
      \bottomrule
    \end{tabular}
    \caption{The final factor: consider the (unique) shortest path between \(\theta\) and \(\beta\) in the Dynkin diagram.}
    \label{table:final-factor}
  \end{table}

  We now systematically analyse the different cases listed in \cref{table:subsets} (\(G\) of classical type):

  \newcommand{\proofcase}[3]{\paragraph{\(#1 \supset #2\) \ifstrempty{#3}{}{\; \(\scriptstyle (#3)\)}}~\\}

  \proofcase{A_n}{A_k}{2<k<n}
  There is only one neighbour \(\beta\), and \(\Delta_\gamma(\theta) = a_\theta + a_{\dual{H}\theta} - b_\beta\).
  All coefficients of all positive roots \(\gamma\) are \(0\) or \(1\).   If \(b_\beta = 0\), then \(\Delta_\gamma(\theta) \geq 0\) for all \(\theta\in\simpleroots{H}\). So suppose now that \(b_\beta = 1\).  If there exists no \(\theta\in\simpleroots{H}\) such that both \(a_\theta = 1\) and \(a_{\dual{H}\theta} = 1\), then clearly \(\Delta_\gamma(\theta) \leq 0\) for all \(\theta\in\simpleroots{H}\). If there exists at least one \(\theta\in \simpleroots{H}\) such that \(a_\theta = 1\) and \(a_{\dual{H}\theta} = 1\), then it follows from the structure of the positive roots of \(A_n\) that for \emph{any} \(\theta\in\simpleroots{H}\) at least one of the coefficients \(a_\theta\), \(a_{\dual{H}\theta}\) is \(1\).  In this case, \(\Delta_\gamma(\theta) \geq 0\) for all \(\theta\in\simpleroots{H}\).


  \proofcase{B_n}{A_k}{2 < k < n}
  There is only one neighbour \(\beta\), and \(\Delta_\gamma(\theta) = a_{\theta} + a_{\dual{H}\theta} - b_\beta\).  (This is true both in the case \(k < n-1\) and in the case \(k = n-1\), which are listed separately in \cref{table:subsets}.)
  All coefficients of all positive roots \(\gamma\) are \(0\), \(1\) or \(2\).
  If \(b_\beta = 0\) or \(b_\beta = 1\), we can argue exactly as in the case \(A_n \supset A_k\).
  Suppose now that \(b_\beta = 2\).  If there exists no \(\theta\in\simpleroots{H}\) such that \(a_\theta + a_{\dual{H}\theta} \geq 3\), then \(\Delta_\gamma(\theta) \leq 0\) for all \(\theta \in\simpleroots{H}\).  Otherwise, if there exists at least one \(\theta\) such that \(a_\theta + a_{\dual{H}\theta} \geq 3\), then there exists some \(\theta\) such that  \(a_\theta \geq 1\) and \(a_{\dual{H}\theta} = 2\), and it follows from the structure of the positive roots of \(B_n\) that in this case \(a_\theta + a_{\dual{H}\theta} \geq 2\) for all \(\theta\in\simpleroots{H}\).  Thus, in this case \(\Delta_\gamma(\theta) \geq 0\) for all \(\theta\in\simpleroots{H}\).

  \proofcase{B_n}{B_k}{1 < k < n}
  There is only one neighbour \(\beta\), and \(\Delta_\gamma(\theta) = 2a_\theta - 2b_\beta\).
  As remarked in the previous case, all coefficients of all positive roots \(\gamma\) are \(0\), \(1\) or \(2\).
  If \(b_\beta = 0\), then \(\Delta_\gamma(\theta) \geq 0\) for all \(\theta \in \simpleroots{H}\).
  If \(b_\beta = 2\), then also \(a_\theta = 2\) for all \(\theta \in \simpleroots{H}\), so \(\Delta_\gamma(\theta) \geq 0\).
  Finally, suppose \(b_\beta = 1\). If there exists no \(\theta \in \simpleroots{H}\) such that \(a_\theta = 2\), then \(\Delta_\gamma(\theta) \leq 0\) for all \(\theta\in\simpleroots{H}\). If there is a \(\theta \in\simpleroots{H}\) such that \(a_\theta = 2\), then it follows from the structure of positive roots of \(B_n\) that \(a_\theta \geq 1\) for all \(\theta\in\simpleroots{H}\), and hence that \(\Delta_\gamma(\theta) \geq 0\) for all \(\theta\in\simpleroots{H}\).


  \proofcase{C_n}{A_k}{2 < k < n-1}
  We can argue exactly as in the case \(B_n\supset A_k\).

  \proofcase{C_n}{A_{n-1}}{}
  There is only one neighbour \(\beta\), and  \(\Delta_\gamma(\theta) = a_\theta + a_{\dual\theta} - 2b_\beta\).
  The coefficient \(b_\beta\) is either \(0\) or \(1\).
  When \(b_\beta = 0\), \(\Delta_\gamma(\theta) \geq 0\).  So suppose \(b_\beta = 1\).  If \(a_{\theta_{n-1}} < 2\), then \(a_\theta + a_{\dual{H}\theta} \leq 2\) for all \(\theta\in\simpleroots{H}\), so \(\Delta_\gamma(\theta) \leq 0\) for all \(\theta\).
  If \(a_{\theta_{n-1}} = 2\), we distinguish two further cases.
  In case there exists some \(\theta\in\simpleroots{H}\) such that \(a_\theta =2\) and \(a_{\dual{H}\theta} = 2\), we find \(a_\theta + a_{\dual{H}\theta} \geq 2\) for all \(\theta\in\simpleroots{H}\), so \(\Delta_\gamma(\theta)\geq 0\) for all \(\theta\).
  In case there exists no \(\theta\in\simpleroots{H}\) such that \(a_\theta =2\), we find \(a_\theta + a_{\dual{H}\theta} \leq 2\) for all \(\theta\in\simpleroots{H}\), so \(\Delta_\gamma(\theta)\leq 0\) for all \(\theta\).

  \proofcase{C_n}{C_k}{1<k<n}
  There is only one neighbour \(\beta\).
  \begin{align*}
    \Delta_\gamma(\theta_i) &= 2a_{\theta_i} - 2b_\beta \quad \text{ for } i\neq k\\
    \Delta_\gamma(\theta_k) &= 2a_{\theta_k} - b_\beta
  \end{align*}
  The possible coefficients for \(\beta\) are \(0\), \(1\) and \(2\).
  When \(b_\beta = 0\), \(\Delta_\gamma(\theta) \geq 0\).
  Suppose \(b_\beta = 1\).  If \(a_\theta = 0\) for some \(\theta\in\simpleroots{H}\), then \(a_{\theta_i} \leq 1\) for all \(i<k\) and \(a_{\theta_k}=0\).  It follows in this case that \(\Delta_\gamma(\theta) \leq 0\) for all \(\theta\in\simpleroots{H}\).
  If on the other hand \(a_{\theta_i} \neq 0\) for all \(\theta\in\simpleroots{H}\), then \(\Delta_\gamma(\theta)\geq 0\) for all \(\theta\in\simpleroots{H}\).
  Finally, suppose \(b_\beta = 2\). In this case, \(a_{\theta_i} = 2\) for all \(i<k\) and \(a_{\theta_k}=1\), so we find that \(\Delta_\gamma(\theta) = 0\) for all \(\theta\in\simpleroots{H}\).


  \proofcase{D_n}{A_k}{2<k<n-2}
  We can argue exactly as in the case \(B_n \supset A_k\).

  \proofcase{D_n}{A_{n-2}}{}
  There are two neighbours \(\beta_1\) and \(\beta_2\), and \(\Delta_\gamma(\theta) = a_\theta + a_{\dual{H}\theta} - b_{\beta_1} - b_{\beta_2}\).  If at least one of \(\beta_{\beta_1}\), \(\beta_{\beta_2}\) is zero, we can argue as in the case \(A_n\supset A_k\).
  Now suppose \(b_{\beta_1} = b_{\beta_2} = 1\).
  If there exists some \(\theta\in\simpleroots{H}\) such that \(a_{\theta} + a_{\dual{H}\theta} \geq 3\), then we find that \(a_{\theta} + a_{\dual{H}\theta} \geq 2\) for all \(\theta \in\simpleroots{H}\), and hence that \(\Delta_\gamma(\theta)\geq 0\) for all \(\theta \in \simpleroots{H}\).  If on the other hand \(a_{\theta} + a_{\dual{H}\theta} \geq 2\) for all \(\theta\in\simpleroots{H}\), then \(\Delta_\gamma(\theta)\leq 2\) for all \(\theta \in \simpleroots{H}\).

  \proofcase{D_n}{A_{n-1}}{}
  There is only one neighbour \(\beta\).
  \begin{align*}
    \Delta_\gamma(\theta_i) &= a_{\theta_i} + a_{\theta_{n-i}} - 2b_\beta  \quad \text{ for } 1 < i < n-1\\
    \Delta_\gamma(\theta_1) = \Delta_\gamma(\theta_{n-1}) &= a_{\theta_1} + a_{\theta_{n-1}} - b_\beta
  \end{align*}
  It suffices to consider the case \(b_\beta = 1\). If \(a_{\theta_{n-1}} = 0\),
  then \(a_{\theta} \leq 1\) for all \(\theta\in\simpleroots{H}\), hence \(\Delta_\gamma(\theta) \leq 0\) for all \(\theta\in\simpleroots{H}\).
  If \(a_{\theta_{n-1}} = 1\) and \(a_{\theta_1} =1\), then we find that \(\Delta_\gamma(\theta_1) > 0\) and \(\Delta_\gamma(\theta_i)\geq 0\) for \(1<i<n-1\).
  If \(a_{\theta_{n-1}} = 1\) and \(a_{\theta_1} = 0\), then \(\Delta_\gamma(\theta_1) = 0\) and we need to distinguish two further cases:
  either there exists some \(i\) with \(1 < i < n-1\) and \(a_{\theta_i} + a_{\theta_{n-i}} > 2\) or \(a_{\theta_i} + a_{\theta_{n-i}} \leq 2\) for all \(i\).
  In the first case, it follows that \(a_{\theta_i}+a_{\theta_{n-i}} \geq 2\) for all \(i\) with \(1<i<n-1\), and hence that \(\Theta_\gamma(\theta) \geq 0\) for all \(\theta\in\simpleroots{H}\).  In the second case, it follows that \(\Delta_\gamma(\theta)\leq 0\) for all \(\theta\in\simpleroots{H}\).

  \proofcase{D_n}{A_3}{}
  This is left as an easy exercise. 

  \proofcase{D_n}{D_k}{3<k<n}
  There is only one neighbour \(\beta\).  For \(\theta\in \{\theta_1,\dots,\theta_{k-1}\}\),
  \begin{align*}
    \Delta_\gamma(\theta) &= 2a_\theta - 2b_\beta\\
    \intertext{For \(\theta\in \{\theta_{k-1},\theta_k\}\),}
    \Delta_\gamma(\theta)
    &= \begin{cases} a_{\theta_{k-1}} + a_{\theta_k} - b_\beta &\text{ if \(k\) is odd}\\2a_{\theta} - b_\beta &\text{ if \(k\) is even} \end{cases}
  \end{align*}
  When \(k\) is even, there is a positive root \(\gamma = \beta + \theta_1 + \theta_2 + \cdots + \theta_{k-1}\) for which \(\Delta_\gamma(\theta_{k-1}) = 1\) but \(\Delta_\gamma(\theta_k) = -1\). Thus, the criterion of \cref{lem:numerical-criterion} fails for even \(k\), which is why this case is excluded in \cref{thm:main}.

  When \(k\) is odd, on the other hand, the criterion is satisfied:
  If \(\gamma\) is a positive root with \(a_{\theta_{k-1}} = a_{\theta_k} = 1\), we find \(\Delta_\gamma(\theta) \geq 0\) for all \(\theta\in\simpleroots{H}\).  For the remaining positive roots, at most one of the two coefficients \(a_{\theta_{k-1}},a_{\theta_k}\) is \(1\) and the other one or two coefficients are \(0\), and we find that \(\Delta_\gamma(\theta) \leq 0\) for all \(\theta\in\simpleroots{H}\).

  This completes the proof for \(G\) of classical type.  The cases listed in \cref{table:results-connected}, where \(G\) is of exceptional type, could be analysed similarly.  We will not bore the reader with this rather tedious analysis, but rather refer to our implementation in Macaulay~2 \cite{code} to complete the proof in these finitely many remaining cases.
\end{proof}

When \(\simpleroots{H}\) has two or more components, condition~\ref{cond:singlecell} typically fails, though we have not investigated this systematically.

\subsection{Condition~\ref{cond:singlecell} implies condition~\ref{cond:orbits}}
\label{sec:singlecell-orbits}
\begin{lem}
  Consider \(I\subset \simpleroots{}\) such that \(\dual{} I = I\). Then \(\cellclosureRR{I}^\dual{}\cap\weightspaceZZ\) is a free abelian monoid with the following basis:
  one basis element \(\omega_\alpha\) for each \(\dual{}\)-self-dual \(\alpha\in\simpleroots{}\setminus I\), and one basis element \(\omega_\alpha+\dual{}\omega_\alpha = \omega_\alpha + \omega_{\dual{}\alpha}\) for each pair of \(\dual{}\)-dual simple roots \(\{\alpha,\dual{}\alpha\}\subset \simpleroots{}\setminus I\) with \(\alpha\neq\dual{}\alpha\).
\end{lem}
\begin{proof}
  As \(G\) is simply-connected, \(\closedWeylchamberRR{}\cap\weightspaceZZ\) is a free abelian monoid on a basis given by the fundamental weights \(\omega_\alpha\) (for \(\alpha\in\simpleroots{}\)) \cite[Theorem~5.62]{adams}.  It follows that, more generally, \(\cellclosureRR{I}\cap\weightspaceZZ\) is a free abelian monoid on a basis given by the fundamental weights \(\omega_\alpha\) with \(\alpha\in\simpleroots{}\setminus I\).
  The involution \(\dual{}\) permutes this basis by the formula \(\dual{}\omega_\alpha = \omega_{\dual{}\alpha}\) for each \(\alpha\in\simpleroots{}\), so the claim follows.
\end{proof}

Assume now that condition~\ref{cond:singlecell} holds, so that \(\closedWeylchamberRR{H}^\dual{H} = w.(\cellclosureRR{I}^\dual{})\) for some \(w\in\Weyl{}\) and some \(I\subset \simpleroots{}\) with \(\dual{} I = I\).
Then, by the above lemma, \(\closedWeylchamberRR{H}^\dual{H}\cap\weightspaceZZ\) has a basis consisting of one element \(w.\omega_\alpha\) for each \(\dual{}\)-self-dual root \(\alpha\in \simpleroots{}\setminus I\) and one element \(w.(\omega_\alpha + \omega_{\dual{}\alpha})\) for each pair of \(\dual{}\)-dual roots in \(\simpleroots{}\setminus I\).  So \ref{cond:orbits} is satisfied.

\subsection{Condition~\ref{cond:orbits} implies condition~\ref{cond:ext}}
\label{sec:orbits-ext}

\begin{lem}\label{lem:symsum{}-decomposition}
  Consider the natural map \(\h^+(\R G)\to \h^+(\R(\levisubgroup{H}))\) induced by restriction.
  For \(\omega\in\closedWeylchamberRR{}\cap\weightspaceZZ\), the class of the symmetric sum \([\symsum{}(\omega)]\) decomposes in \(\h^+(\R(\levisubgroup{H}))\) as a sum of classes of symmetric sums \([\symsum{H}(\tau)]\) indexed by the intersection points \(\tau\) of the Weyl orbit \(\Weyl{}.\omega\) with \(\closedWeylchamberRR{H}^\dual{H}\).
  In particular, in \(\h^+(\R(\levisubgroup{H}))\),
  \[
    [\symsum{}(\omega)]=
    \begin{cases}
      [\symsum{H}(\tau)] & \text{ if } \Weyl{}.\omega\cap \closedWeylchamberRR{H}^\dual{H} = \{\tau\}\\
      0 & \text{ if } \Weyl{}.\omega\cap \closedWeylchamberRR{H}^\dual{H} = \emptyset
    \end{cases}
  \]
\end{lem}
\begin{proof}
  As \(\closedWeylchamberRR{H}\) is a fundamental domain for the action of \(\Weyl{H}\) on \(\weightspaceRR\) (\cite[Corollary~5.16]{adams}), the Weyl orbit \(\Weyl{}.\omega\) decomposes as a disjoint union over the orbits \(\Weyl{H}.\tau\) for the distinct elements \(\tau\in\Weyl{}.\omega\cap\closedWeylchamberRR{}\).
  This implies that,  for any \(\omega\in\closedWeylchamberRR{}\),
  \(
  \symsum{}(\omega)=\textstyle\sum_{\tau\in \Weyl{}.\omega\cap \closedWeylchamberRR{H}} \symsum{H}(\tau)
  \).
  In Tate cohomology, we see only those symmetric sums \(\symsum{H}(\tau)\) with \(\tau\in\closedWeylchamberRR{H}^\dual{H}\); the other symmetric sums occur in dual pairs and disappear.
\end{proof}

\begin{prop}
  \label{prop:Tate-of-RH-under-orbit-condition}
   If condition~\ref{cond:orbits} holds with respect to \(I\subset \simpleroots{}\), then \(\h^*(\R(\levisubgroup{H}))\) is freely generated as a \(\ZZII\)-graded \(\ZZII\)-algebra by the following elements:
  \begin{alignat*}{7}
      &i^*[\rsymsum{}(\omega_\alpha)] &&\text{ for each \(\alpha\in\simpleroots{}\setminus I\) with \(\dual{}\alpha = \alpha\), and}\\
      &i^*[\rsymsum{}(\omega_\beta + \omega_{\dual{}\beta})] &&\text{ for each \(\beta\in\simpleroots{}\setminus I\) with \(\dual{}\beta\neq \beta\)}.
    \end{alignat*}
\end{prop}
\begin{proof}
   Assuming condition~\ref{cond:orbits}, \cref{lem:symsum{}-decomposition} shows that the elements specified in the \namecref{prop:Tate-of-RH-under-orbit-condition} are of the form \([\rsymsum{H}(\tau_1)],\dots,[\rsymsum{H}(\tau_a)]\) for a basis of \(\closedWeylchamberRR{H}^\dual{H}\cap\weightspaceZZ\).  Thus, the claim is immediate from \cref{lemmatatecohfreemonoid}.
\end{proof}

\begin{remark}
  \label{rem:Tate-of-RH-under-singlecell}
  Under the stronger condition \ref{cond:singlecell}\(_I\), we can also say something about the remaining generators of \(\R(G)\):
  \begin{alignat*}{7}
    &i^*[\rsymsum{}(\omega_\alpha)] &= 0 &\text{ for each \(\alpha\in I\) with \(\dual{}\alpha = \alpha\), and} \\
    &i^*[\rsymsum{}(\omega_\beta+\omega_{\dual{}\beta})] &= 0 & \text{  for each \(\beta\in I\) with \(\dual{}\beta\neq\beta\)}
  \end{alignat*}
  in \(\h^*(\R(\levisubgroup{H}))\).  Indeed, the weights \(\omega_\alpha\) and \(\omega_\beta + \omega_{\dual{}\beta}\) with \(\alpha,\beta \in I\) do not lie in \(\cellclosureRR{I}\), and hence their Weyl orbits cannot intersect \(\closedWeylchamberRR{H}^\dual{H} = w(\cellclosureRR{I}^\dual{})\).  We may therefore again conclude by applying \Cref{lem:symsum{}-decomposition}.  (Note that \([\symsum{}(\omega)] = 0 \) implies, a fortiori, that \([\rsymsum{}(\omega)] = 0\).)
\end{remark}

The following \namecref{prop:key-Tate-argument} now indicates how to prove \cref{thm:main}:
\begin{lem}\label{prop:key-Tate-argument}
  Consider a ring with involution \((R,\circ)\) together with an ideal \(\mathfrak g\) that is preserved by the involution.  Suppose that \(\h^-(R,\circ) = 0\), and that \(\ideal g\) has the following structure:
  \begin{compactenum}
  \item The ideal \(\mathfrak g\) is generated by self-dual elements  \(\tau_1,\dots,\tau_a,\lambda_1,\dots,\lambda_b\) that form a regular sequence in \(R\).
  \item The classes \([\tau_1],\dots,[\tau_a]\) form a regular sequence in \(\h^+(R)\).
  \item Each of the classes \([\lambda_i]\) is contained in the ideal in \(\h^+(R)\) generated by the classes \([\tau_1],\dots,[\tau_a]\).
  \end{compactenum}
  Then the canonical ring homomorphism \(\h^*(R/(\tau_1,\dots,\tau_a)) \leftarrow \h^*(R)/\left([\tau_1],\dots,[\tau_a]\right)\) is an isomorphism, and we have an isomorphism of graded abelian groups
  \[
    \h^*(R/\ideal g) \cong \left.\h^*(R)\middle/\left([\tau_1],\dots,[\tau_a]\right)\right. \otimes \bigwedge_{i = 1,\dots, b} [u_i]
  \]
  for certain classes \([u_i] \in \h^-(R/\ideal g)\).  In fact, this is an isomorphism of \(\h^*(R/(\tau_1,\dots,\tau_a))\)-modules.
  As representatives \(u_i\), we may pick arbitrary elements \(u_i \in R\) such that \(u_i+u_i^\circ \equiv \lambda_i \mod \mathfrak g\).
\end{lem}
\begin{proof}
The first claim, concerning \(R/(\tau_1,\dots,\tau_a)\), is \cite[Corollary~1.5]{hemmert1}.  It follows from repeated application of \cite[Lemma~1.4 (iii)]{hemmert1}.  The claims concerning \(R/\mathfrak g\) then follow from repeated application of \cite[Lemma~1.4 (i)]{hemmert1}.
\end{proof}

\begin{proof}[Proof of \cref{thm:main}]
  Given \eqref{eq:main-iso-2}, we need to compute the Tate cohomology of \(R/\mathfrak g\), where \(R:= \R(\levisubgroup{H})\) and \(\mathfrak g\) is the ideal generated by the following \(\fixrank{}\) elements:
  \(\rsymsum{}(\omega_\alpha)\) for each \(\alpha\in\simpleroots{}\) with \(\dual{}\alpha = \alpha\), and \(\rsymsum{}(\omega_\beta + \omega_{\dual{}\beta})\) for each \(\beta\in\simpleroots{}\) with \(\dual{}\beta \neq \beta\).  These generators form a regular sequence in \(R\) because they form a regular sequence in \(\R G\), and because \(R\) is free as an \(\R G\)-module. By \cref{prop:Tate-of-RH-under-orbit-condition}, a subset of these generators indexed by \(\simpleroots{}\setminus I\) freely generates \(\h^+(R)\) as a commutative ring, and hence satisfies condition~2 of \cref{prop:key-Tate-argument}.
  Moreover, the quotient of \(\h^+(R)\) by the classes of these generators is trivial, equal to \(\ZZII\).  The remaining \(\fixrank{}-\fixrank{H}\) generators of \(\mathfrak g\) -- all of rank zero by construction -- therefore satisfy condition~3 of \cref{prop:key-Tate-argument}.  It thus follows from \cref{prop:key-Tate-argument} that \(\h^*(R/\mathfrak g)\) is an exterior algebra on \(\fixrank{}-\fixrank{H}\) generators, or at least that it is isomorphic to such an exterior algebra as a graded group.  It is indeed isomorphic as an algebra by \cref{prop:odd-degre-squares-are-zero} below.
\end{proof}

\begin{prop}[{\cite[Proposition~4.6]{hemmert:thesis}}]
  \label{prop:odd-degre-squares-are-zero}
  Let \(X\) be a finite cell complex with \(\K^1(X) = 0\).  The square of any element of odd degree in \(\Witt^*(X)\) is zero. \qed
\end{prop}

\section{Proof of \cref{thm:main-degrees}}
\label{sec:main-degrees}

\subsection{Characterization of maximal cells}
We begin with an equivalent characterization of the subsets \(I\subset \simpleroots{}\) that appear in the decomposition \eqref{eq:fixed-space-decomposition-2} of the fixed-point space \((\weightspaceRR)^\dual{H}\) into closed cells.

\begin{lem}\label{lem:fixed-cells-2}
  The following conditions on a subset \(I\subset \simpleroots{}\) are equivalent:
  \begin{enumerate}[(a')]
  \item For some \(w\in\Weyl{}\), the pair \((I,w)\) satisfies the conditions of \cref{lem:fixed-cells} and also \eqref{eq:fixed-cells-2}:
    \begin{equation*}
      \cardinality{\frac{I}{\dual{}}} = \cardinality{\frac{\simpleroots{}}{\dual{}}} - \cardinality{\frac{\simpleroots{H}}{\dual{H}}}
    \end{equation*}
  \item \(I\) satisfies each of the following conditions:
  \begin{compactenum}[(i)]
  \item \(\dual{} I = I\)
  \item The involutions \(\dual{}\) and \(\dual{I}\) agree on \(I\)
  \item \(\longest{H}\) is conjugate to \(\dual{} \dual{I}=\longest{} \longest{I}\)
  \end{compactenum}
  \end{enumerate}
\end{lem}
\begin{proof}
  (a' \(\Rightarrow\) b') Suppose \((I,w)\) satisfies the conditions of \cref{lem:fixed-cells} and \eqref{eq:fixed-cells-2}.  We first show that \(w^{-1}\dual{H} w\) commutes with \(\dual{}\).   Decompose \(\weightspaceRR\) into the following orthogonal subspaces:
  \begin{align*}
    \RR I &\;= \RR\{\alpha \mid \alpha \in I \} \\
    \Omega_I &:= \RR\{\omega_\beta \mid \beta \in \simpleroots{}\setminus I\}
  \end{align*}
  By part~(c) of \cref{lem:fixed-cells}, \(\dual{}\) and \(w^{-1}\dual{H} w\) can both be restricted to \(\Omega_I\), and they agree on \(\Omega_I\).  As they are orthogonal transformations, they can also be restricted to the orthogonal complements \(\RR I\).  It suffices to show that these restrictions to \(\RR I\) commute.  From what we have just stated, we immediately obtain the following equalities for the dimensions of fixed-point spaces:
  \begin{alignat}{7}
    \label{eq:CUXVALQ1}
    &\plusdim(\dual{}_{|\RR I}) &\;+\;&\plusdim(\dual{}_{|\Omega_I}) &\;=\;&\plusdim(\dual{}) \\
    \label{eq:CUXVALQ2}
    &\plusdim(w^{-1}\dual{H} w_{|\RR I}) &\;+\;&\plusdim(w^{-1}\dual{H} w_{|\Omega_I}) &\;=\;&\plusdim(\dual{H}) \\
     &&& \plusdim(w^{-1}\dual{H} w_{|\Omega_I}) &\;=\;&\plusdim(\dual{}_{|\Omega_I})
  \end{alignat}
  Altogether, this implies \eqref{eq:XWNGA1} below.  On the other hand, \eqref{eq:fixed-cells-2} can be translated to \eqref{eq:XWNGA2} below:
  \begin{alignat}{7}
    \label{eq:XWNGA1}
    &\plusdim(\dual{}_{|\RR I}) &\;-\;&\plusdim(w^{-1}\dual{H} w_{|\RR I}) &\;=\;& \plusdim(\dual{}) - \plusdim(\dual{H})  \\
    \label{eq:XWNGA2}
    &\plusdim(\dual{}_{|\RR_I}) &&&\;=\;& \plusdim(\dual{}) - \plusdim(\dual{H})
  \end{alignat}
  In combination, this shows that \(\plusdim(w^{-1}\dual{H} w_{|\RR I}) = 0\), and thus that:
  \begin{equation}\label{eq:WATERCRACK}
    w^{-1}\dual{H} w_{|\RR I} = -\id_{\RR I}
  \end{equation}
  In particular, \(w^{-1}\dual{H} w\) commutes with \(\dual{}\), as claimed.

  Now let \(\Weylo\subset\Weyl{}\) denote the subgroup of all elements commuting with \(\dual{}\).
  By \cite[1.32\,(b)]{steinberg:endomorphisms}, we can identify \(\Weylo\) with the Weyl group \(\Weylfolded\) of a ``folded'' root system \(\folded{\simpleroots{}}\), whose simple roots correspond to the \(\dual{}\)-orbits of the simple roots in \(\simpleroots{}\).  Let \(\iota\colon \Weylfolded \hookrightarrow \Weyl{}\) denote the embedding that induces the identification.  We will need the following two facts that can be deduced from Steinberg's explicit description of the folded root system (see \cite[Propositions 1.3 \& 1.4]{aux:involutions}):
  \begin{itemize}
  \item[(F1)]
    \(\Weylo\cap\Weyl{I} = \iota(\Weylfoldedparabolic{\folded{I}})\), where \(\Weylfoldedparabolic{\folded{I}}\) denotes the parabolic subgroup of \(\Weylfolded\) corresponding to the folded subset \(\folded{I}\subset\folded{\simpleroots{}}\).
  \item[(F2)]
    Given any subset \(J\subset\simpleroots{}\) such that \(\dual{}J = J\), the longest element of the parabolic subgroup \(\Weylfoldedparabolic{\folded{J}}\) corresponds to the longest element of the parabolic subgroup \(\Weyl{J}\), i.e. \(\iota(\longest{\folded{J}}) = \longest{J}\).\\
    (The special case when \(J\) is a \(\dual{}\)-orbit is \cite[1.30\,(b)]{steinberg:endomorphisms}.)
  \end{itemize}
  From the above, we find that \(u := \dual{} w^{-1}\dual{H} w\in\Weylo \cap \Weyl{I}\). Hence by (F1), there is a unique
  \(v\in\Weylfoldedparabolic{\folded{I}}\) such that  \(\iota v = u\).  Any involution in a Weyl group is conjugate to the longest element of some parabolic subgroup.  (This follows, for example, from \cite[\S\,28-5\,(i)]{kane}.)  Therefore, in \(\Weylfolded\), \(v\) is conjugate to the longest element \(\longest{J'}\) of some parabolic subgroup \(\Weylfoldedparabolic{J'}\), for some subset \(J'\subset \folded{\simpleroots{}}\) .  We can lift the subset to a subset \(J\subset \simpleroots{}\) such that \(\dual{}J = J\) and such that \(J'=\folded{J}\).  Then, by (F2), \(\iota(\longest{J'}) = \longest{J}\).  It follows that \(u\) is conjugate to \(\longest{J}\) by some \(w'\in\Weylo\), i.e.\ \emph{by some \(w'\) that commutes with \(\dual{}\)}.    We therefore find:
  \begin{equation}\label{eq:KATDUKE}
    \dual{} (ww')^{-1}\dual{H} w w' = \longest{J}
  \end{equation}
  This shows in particular that the pair \((J,ww')\) satisfies condition~(c) of \cref{lem:fixed-cells}. As we assumed \(\cardinality{I/\dual{}}\) to be minimal, this implies \(J = I\).  Equation~\eqref{eq:KATDUKE} moreover shows that \(\longest{H}\) is conjugate to \(\dual{}\dual{I}\), as claimed.  Finally, applying the argument leading up to \(\eqref{eq:WATERCRACK}\) to \((I,ww')\) in place of \((I,w)\), we find that \(-(ww')^{-1}\dual{H} ww'\) acts as the identity on \(I\).  In combination with \eqref{eq:KATDUKE}, it follows that \(\dual{I}\) acts on \(I\) in the same way as \(\dual{}\) does.

  (a' \(\Leftarrow\) b') By assumption, we can find an element \(w\in\Weyl{}\) such that \(\longest{H} = w\dual{} \dual{I} w^{-1}\).  Then \(\dual{} w^{-1}\dual{H} w = \longest{I}\).  So \((I,w)\) satisfies condition~(b) of \cref{lem:fixed-cells}. For the minimality condition, we first note that \(\plusdim(\dual{}_{|\Omega_I}) = \plusdim(\dual{} \longest{I})\), as \(\longest{I}\) acts as the identity on \(\Omega_I\) and as \(\dual{} \longest{I}\) acts as \(-\id\) on \(\RR I\).  As \(\dual{}\longest{I}\) is conjugate to \(\dual{H}\), we may further deduce that \(\plusdim(\dual{}_{|\Omega_I}) = \plusdim(\dual{H})\).  Now \eqref{eq:fixed-cells-2} follows from \eqref{eq:CUXVALQ1}.
\end{proof}

\subsection{Generators of real and quaternionic types}
\begin{defn}
  An element of \(\R(G)\) is said to be of \textbf{real type} or of \textbf{quaternionic type} if it lies in the image of \(c\) or \(c'\), respectively.
\end{defn}

Note that a self-dual element is of one of these types if and only if it is \textbf{Witt homogeneous}, in the sense that it defines a homogeneous element of \(\Witt^0(\R(G))\oplus \Witt^2(\R(G))\) under the isomorphism
\begin{equation}\label{eq:Bousfield-for-reps}
  \Witt^0(\R(G))\oplus \Witt^2(\R(G)) \xrightarrow[c\oplus c']{\cong} \h^+(\R G)
\end{equation}
that is analogous to \eqref{eq:Bousfield} (see \cite[(2.1)]{zibrowius:koff}).  Any self-dual irreducible complex representation is of precisely one of these types.

\begin{lem}\label{lem:real-sums}
  A self-dual complex representation of a compact Lie group is of real type if and only if, in its decomposition into irreducible representations, all summands of quaternionic type occur with even multiplicities.   It is of quaternionic type if and only if all summands of real type occur with even multiplicities.
\end{lem}
\begin{proof}
  This follows from the isomorphism \eqref{eq:Bousfield-for-reps} and the fact that the self-dual irreducible complex representations from a \(\ZZII\)-basis of \(\h^+(\R(G))\).
\end{proof}

Recall that in \(\R(G)\), the irreducible representation \(\irreducibleRepresentation{\omega}\) corresponding to a dominant weight \(\omega\) can be written as a sum
\begin{equation}\label{eq:irred-as-symsum{}s}
  \irreducibleRepresentation{\omega}=\symsum{}(\omega) + \textstyle\sum_{\omega'}a_{\omega'}\symsum{}(\omega'),
\end{equation}
for certain integer coefficients \(a_{\omega'}\) that are non-zero only for dominant weights \(\omega'< \omega\) and congruent to \(\omega\) modulo the root lattice.\footnote{
  The congruence of the relevant weights modulo the root lattice is discussed for irreducible representations of complex Lie algebras below Remark~14.3 in \cite{fultonharris}.  The discussion below Exercise~23.6 of \cite{fultonharris} shows that it also holds for irreducible representations of complex Lie groups, and hence also for irreducible representations of real \emph{compact} Lie groups.
  }
It follows that, conversely, the symmetric sum \(\symsum{}(\omega)\) can be written as
\begin{equation}\label{eq:symsum{}-as-irreds}
  \symsum{}(\omega) = \irreducibleRepresentation{\omega} + \textstyle\sum_{\omega'}b_{\omega'}\irreducibleRepresentation{\omega'},
\end{equation}
for certain integer coefficients \(a_{\omega'}\) that are non-zero only for dominant weights \(\omega'\leq \omega\) and congruent to \(\omega\) modulo the root lattice.

\begin{prop}\label{prop:real-iff-odd}
  Suppose the irreducible representation \(\irreducibleRepresentation{\omega}\) of \(G\) with highest weight \(\omega\) is self-dual  (i.e.\ \(\omega\) dominant such that \(\dual{}\omega=\omega\)). Write \(\rho^\vee\) for half the sum of all positive coroots of \(G\).   The representation \(\irreducibleRepresentation{\omega}\) is of real type if and only if \(\pairing{2\rho^\vee}{\omega}\) is even, and of quaternionic type if and only if \(\pairing{2\rho^\vee}{\omega}\) is odd.
\end{prop}
This is well-known. See for example \cite[Chapter~IX, §7.2, Proposition~1]{bourbaki789}, where the statement is deduced from an analogous statement about representations of Lie algebras.  However, as we have not found a discussion of this result in standard references on compact Lie groups such as \cite{adams,btd}, we include a more direct proof here.
\begin{proof}
  Consider the principal \(\SU(2)\) of \(G\) \cite[Theorem~10]{vogan:challenges}\cite[\S\,5]{kostant:principal}.  This is a homomorphism \(\phi\colon \SU(2)\to G\) whose restriction to (chosen) maximal tori is given by \(2\rho^\vee\colon S^1\to T\).  Write \(\omega_1\) for the fundamental weight of \(\SU(2)\), and \(\irreducibleRepresentation{0}\) for the irreducible representation of \(\SU(2)\) with highest weight \(\pairing{2\rho^\vee}{\omega}\omega_1\).   We claim \(\irreducibleRepresentation{0}\) occurs with multiplicity one in the decomposition of \(\phi^*\irreducibleRepresentation{\omega}\) into irreducible representations.  This can be seen as follows.  As remarked under \eqref{eq:irred-as-symsum{}s} above, the weights \(\omega'\) that occur with non-zero multiplicity in \(\irreducibleRepresentation{\omega}\) have the form \(\omega' = \omega - \beta\) for some positive root \(\beta\) of \(G\).   The corresponding weight space pulls back to a weight space for the weight
  \[
    \pairing{2\rho^\vee}{\omega'}\omega_1 = \pairing{2\rho^\vee}{\omega}\omega_1 -\pairing{2\rho^\vee}{\beta}\omega_1.
  \]
  As \(\pairing{\rho^\vee}{\alpha} = 1\) for each simple root \(\alpha\), this shows that the weight \(\pairing{2\rho^\vee}{\omega}\omega_1\) has multiplicity one in \(\phi^*\irreducibleRepresentation{\omega}\).  So \(\irreducibleRepresentation{0}\) occurs with multiplicity one, as claimed.

  Now if \(\irreducibleRepresentation{\omega}\) is real, then so is \(\phi^*\irreducibleRepresentation{\omega}\), and hence so is \(\irreducibleRepresentation{0}\), by \cref{lem:real-sums}.
  By the same argument, if  \(\irreducibleRepresentation{\omega}\) is quaternionic, also \(\irreducibleRepresentation{0}\).  As any self-dual irreducible representation is of precisely one of these types, the claim follows.
\end{proof}

\begin{cor}\label{cor:symsum{}s-witt-homogeneous}
  For any self-dual dominant weight \(\omega\), the symmetric sum \(\symsum{}(\omega)\) is Witt homogeneous.  It is of the same type as the irreducible representation \(\irreducibleRepresentation{\omega}\) of highest weight \(\omega\).
\end{cor}
\begin{proof}
  For any root \(\beta\), \(\pairing{2\rho^\vee}{\beta}\) is even.  Together with \cref{prop:real-iff-odd}, this shows that irreducible representations \(\irreducibleRepresentation{\omega'}\) occurring in \eqref{eq:symsum{}-as-irreds} are of the same type as \(\irreducibleRepresentation{\omega}\).  So \(\symsum{}(\omega)\) is Witt homogeneous.  It follows from \cref{lem:real-sums} and \eqref{eq:symsum{}-as-irreds} that it is of the same type as \(\irreducibleRepresentation{\omega}\).
\end{proof}

\begin{remark}
  For \emph{miniscule} weights \(\omega_\alpha\), the fundamental representations \(\fundamentalRepresentation{\alpha}\) are equal to the symmetric sums \(\symsum{}(\omega_\alpha)\), i.e.\ all additional summands in \eqref{eq:irred-as-symsum{}s} vanish.   In type \(A_n\), all fundamental weights are miniscule, but in general only few fundamental weights are.
\end{remark}

\subsection{The degrees of the generators of the Witt ring}

The identification of the degrees in the main theorem rests on:
\begin{lem}
  \label{lem:degrees}
  Consider an element \(u\in\R(\levisubgroup{H})\) that can be written as
  \[
    u+u^* = \sum_{j=1}^{k} \mu_j \cdot i^*\lambda_j
  \]
  for certain self-dual elements \(\mu_j\in\R(\levisubgroup{H})\) and \(\lambda_j\in \reducedR G\).
  Any such element gives rise rise to an element \([\alpha(u)]\in\h^-(\K(G/\levisubgroup{H}))\).
  If all \(\mu_j\) and \(\lambda_j\) are of real type, then \([\alpha(u)]\) corresponds to an element of \(\Witt^3(\K(G/\levisubgroup{H}))\) under the isomorphism~\eqref{eq:Bousfield}.
  If for each \(j\), precisely one of \(\mu_j,\lambda_j\) is of real type and the other is of quaternionic type, then \([\alpha(u)]\in\Witt^1(\K(G/\levisubgroup{H}))\) under the isomorphism~\eqref{eq:Bousfield}.
\end{lem}
\begin{proof}
  This is a mild generalization of \cite[Lemma~3.6]{hemmert1}.  For completeness, we sketch a proof of the last statement.
  Consider the following commutative diagram:
  \[
    \begin{tikzcd}
      & \R(\levisubgroup{H}) \arrow[d, "q"] \arrow[r, "\alphaU"]           & \K(G/\levisubgroup{H}) \arrow[d, "q"] \\
      \reducedRSp(G) \arrow[r, "i^*"] \arrow[d, "c'", hook] & \RSp(\levisubgroup{H}) \arrow[r, "\alphaSp"] \arrow[d, "c'", hook] & \KO^4(G/\levisubgroup{H})               \\
      \reducedR(G) \arrow[r, "i^*"]                       & \R(\levisubgroup{H})                                             &
    \end{tikzcd}
  \]
  By assumption, \(\mu_j\cdot i^*(\lambda_j) = c'(\hat\mu_j \cdot i^*\hat\lambda_j)\) for certain elements \(\hat\mu_j\) and \(\hat\lambda_j\) in \(\RSp(\levisubgroup{H})\) and \(\reducedRO(G)\), respectively, or in \(\RO(\levisubgroup{H})\) and \(\reducedRSp(G)\), respectively.
  Note also that \(c'q = \id + *\).  Thus, \(c'q(u) = u+u^* = c'(\sum_j \hat\mu_j \cdot i^*\hat\lambda_j)\).   As \(c'\) is injective on representation rings of compact Lie groups, it follows that \(q(u) = \sum_j \hat\mu_j \cdot i^*\hat\lambda_j\).  So \(q(\alpha u) = \alphaSp(q(u)) = \sum_j\alphaSp(\hat\mu_j \cdot i^*\hat\lambda_j)\), and this is zero by the multiplicative properties of \(\alphaO \oplus \alphaSp\), and by the fact that \(\alphaO i^* = 0\) and \(\alphaSp i^* = 0\).  Thus, \(\alpha u \in\ker(q)\).  It follows that \([\alphaU(u)]\) corresponds to an element of \(\Witt^1(G/\levisubgroup{H})\).   (See the more precise description  of the isomorphism \eqref{eq:Bousfield} in \cite[\S\,1.2]{zibrowius:koff}).
\end{proof}

\begin{proof}[Proof of \cref{thm:main-degrees}]
  We have seen in the proof of \cref{thm:main} that \(\h^*(\R(\levisubgroup{H}))\) is generated by certain elements \([\tau_{[\alpha]}]\) indexed by \(\alpha \in (\simpleroots{}\setminus I)/\dual{}\).
  Let us again write \(\mathfrak g\) for the ideal generated by the elements \(\tau_{[\alpha]}\).
  We have moreover seen that \(\h^*(\K(G/\levisubgroup{H}))\) is generated by certain elements  \([\alphaU(u_{[\gamma]})]\) indexed by \(\gamma \in I/\dual{}\).
  By construction,
  \begin{align*}
    u_{[\gamma]} + \dual{\levisubgroup{H}} u_{[\gamma]} &\equiv \begin{cases}
      i^*\rsymsum{}(\omega_\gamma)  & \text{ if } \dual{}\gamma = \gamma\\
      i^*\rsymsum{}(\omega_\gamma + \omega_{\dual{}\gamma}) & \text{ if } \dual{}\gamma \neq \gamma
    \end{cases}
  \end{align*}
  in \(\R(\levisubgroup{H})/\mathfrak g\).  By \cref{rem:Tate-of-RH-under-singlecell}, we may in fact choose the elements \(u_{[\gamma]}\) such that these equalities hold in \(\h^*(\R(\levisubgroup{H}))\), not just modulo \(\mathfrak g\).
  By \cref{cor:symsum{}s-witt-homogeneous}, the Tate cohomology classes \([\rsymsum{}(\omega_\alpha)]\) and \([\rsymsum{}(\omega_\gamma + \omega_{\dual{}\gamma})]\) are Witt homogeneous.
  We may therefore conclude using \cref{lem:degrees}:  take \(u := u_{[\gamma]}\), \(k=1\), \(\mu_1 = 1\), and \(\lambda_1 = \rsymsum{}(\omega_\gamma)\) or \(\lambda_1 = \rsymsum{}(\omega_\gamma + \omega_{\dual{}\gamma})\).
\end{proof}

\section{Proof of \cref{thm:F4}}
\label{sec:F4}

The Dynkin diagram of $F_4$ is given by
\begin{center}
  \dynkin[mark=o,label]{F}{4}
\end{center}
We enumerate a choice of simple roots as indicated, and write \(\omega_i\) for the corresponding fundamental weights.  This is the same numbering as in \cite[Plate~VIII]{bourbaki456}.

\begin{table}
  \newcommand{\generatorlist}[1]{\substack{~\\#1\\~}}
  \begin{adjustwidth}{-4cm}{-4cm}
    \centering
  \begin{tabular}{CCCCCCc}
    \toprule
    H & \dual{H}\omega_1 & \dual{H}\omega_2 & \dual{H}\omega_3 & \dual{H}\omega_4 & \substack{\text{generators}\\\text{of } \fixmonoid} & \text{free}\\
    \midrule
    \dynkin F{***o} & \omega_1-2\omega_4 & \omega_2-4\omega_4 & \omega_3-3\omega_4 & -\omega_4 & \generatorlist{\omega_1-\omega_4\\\omega_2-2\omega_4\\2\omega_3-3\omega_4} & yes\\
    \dynkin F{o***} & -\omega_1 & \omega_2-3\omega_1 & \omega_3-2\omega_1 & \omega_4-\omega_1 & \generatorlist{2\omega_2-3\omega_1\\2\omega_4-\omega_1\\\omega_2+\omega_4-2\omega_1\\\omega_3-\omega_1}&no\\
    \dynkin F{*o*o} & \omega_1-\omega_2 & -\omega_2 & \omega_3-\omega_2-\omega_4 & -\omega_4 & \generatorlist{2\omega_1-\omega_2\\2\omega_3-\omega_2-\omega_4}& yes\\
    \dynkin F{**o*} & \omega_2-2\omega_3 & \omega_1-2\omega_3 & -\omega_3 & \omega_4-\omega_3 & \generatorlist{\omega_1+\omega_2-2\omega_3\\2\omega_4-\omega_3} &yes\\
    \dynkin F{*o**} & \omega_1-\omega_2 & -\omega_2 & \omega_4-\omega_2 & \omega_3-\omega_2 & \generatorlist{2\omega_1-\omega_2\\\omega_3+\omega_4-\omega_2} &yes\\
    \bottomrule
  \end{tabular}
\end{adjustwidth}
\caption{In the first column, we indicate in black the subset \(\simpleroots{H}\subset \simpleroots{}\). In columns 2--5, the duals of the fundamental weights under the duality induced by the longest element in $\Weyl{H}$ are given. In column~6, we display a set of generators of the abelian monoid $\fixmonoid$, and in column~7 we indicate whether this abelian monoid is free.}
\label{tablef4_1}
\end{table}
\begin{table}
  \begin{adjustwidth}{-4cm}{-4cm}
    \centering
    \renewcommand{\arraystretch}{1.5}
    \begin{tabular}{CCCCC}
      \toprule
      \simpleroots{H} & \closedWeylchamberZZ{H}^\dual{H}\cap \Weyl{}.\omega_1 & \closedWeylchamberZZ{H}^\dual{H}\cap \Weyl{}.\omega_2 & \closedWeylchamberZZ{H}^\dual{H}\cap \Weyl{}.\omega_3 & \closedWeylchamberZZ{H}^\dual{H}\cap \Weyl{}.\omega_4 \\
      \midrule
      \dynkin F{***o} & \left\{\omega_2-2\omega_4\right\} & \left\{\omega_1+2\omega_3-4\omega_4\right\} & \left\{2\omega_3-3\omega_4\right\} & \left\{\omega_1-\omega_4\right\} \\
      \dynkin F{o***} & \left\{2\omega_4-\omega_1\right\} & \left\{2\omega_2-3\omega_1\right\} & \left\{\omega_2+\omega_4-2\omega_1\right\} & \left\{\omega_3-\omega_1\right\} \\
      \dynkin F{*o*o} & \left\{2\omega_1-\omega_2\right\} & \left\{2\omega_1+4\omega_3-3\omega_2-2\omega_4\right\} & \left\{2\omega_1+2\omega_3-2\omega_2-\omega_4\right\} & \left\{2\omega_3-\omega_2-\omega_4\right\} \\
      \dynkin F{**o*} & \left\{\omega_1+\omega_2-2\omega_3\right\} & \left\{\omega_1+\omega_2+4\omega_4-4\omega_3\right\} & \left\{\omega_1+\omega_2+2\omega_4-3\omega_3\right\} & \left\{2\omega_4-\omega_3\right\} \\
      \dynkin F{*o**} & \left\{2\omega_1-\omega_2\right\} & \left\{2\omega_1+2\omega_3+2\omega_4-3\omega_2\right\} & \left\{2\omega_1+\omega_3+\omega_4-2\omega_2\right\} & \left\{\omega_3+\omega_4-\omega_2\right\} \\
      \bottomrule
    \end{tabular}
  \end{adjustwidth}
  \caption{The intersections of the Weyl orbits of the fundamental weights of $F_4$ with the lattice $\fixmonoid$ of self-dual weights in $\closedWeylchamberZZ{H}$.}
\label{tablef4_2}
\end{table}

The first column of Table~\ref{tablef4_1} lists a complete set of representatives of the distinct \(\Weyl{}\)-equivalence classes of those subsets \(\simpleroots{H}\subset\simpleroots{}\) that are not covered by \cref{thm:main}.  In each case, the following columns display the \(\dual{H}\)-duals of the fundamental weights are displayed.  Given the explicit action of \(\dual{H}\) on the fundamental weight, we can easily compute a basis, or at least generators, of \(\fixmonoid\).  These generators are displayed in the final column of the table.  We see that in all cases except $\dynkin F{o***}$, the monoid $\fixmonoid$ is free commutative with the given generators as a basis.  Table~\ref{tablef4_2} displays the intersections of the Weyl orbits of the fundamental weights with \(\fixmonoid\).  In each case, we find that this intersection is a singleton containing one of the generators given in Table~\ref{tablef4_1}.  Thus, \ref{cond:orbits} is satisfied in all cases but one.  It remains to deal with this exceptional case:
\begin{prop}\label{lemmaf4exceptionalcase}
For the subset \({H}\) marked black in \(\dynkin F{o***}\), the Tate cohomology \(\h^*\left(\R(\levisubgroup{H}) / (\rsymsum{}(\omega_i) \mid i \in \{1,2,3,4\})\right)\) is an exterior algebra on a single generator of odd degree.
\end{prop}
\begin{proof}
  Let \(M\) denote the monoid \(\fixmonoid\).  We may deduce from the generators of \(M\) displayed in \cref{tablef4_1} that
  \[
    M = \frac{\NN\{\tau_1,\dots,\tau_4\}}{2\tau_3 = \tau_1 + \tau_2}
  \]
  with
  \begin{align*}
    \tau_1 &:= 2\omega_4-\omega_1 & \quad 
    \tau_2 &:= 2\omega_2-3\omega_1 & \quad 
    \tau_3 &:= \omega_2 + \omega_4 -2\omega_1 & \quad 
    \tau_4 &:= \omega_3 - \omega_1 & \quad 
  \end{align*}
  The monoid ring \(\ZZII[M]\) can therefore be described explicitly as the ring
  \[
    \ZZII[M] = \frac{\ZZII[e^{\tau_1},\dots,e^{\tau_4}]}{(e^{\tau_3})^2-e^{\tau_1}e^{\tau_2}}
  \]
  On the other hand, we know from \cref{prop:Tate-of-H} that \(\h^*(\R H)\) has a \(\ZZII\)-basis given by the elements \(e_{\tau} := [\symsum{H}(\tau)]\) indexed by all \(\tau\in M\).  Let us write \(e_i\) for \(e_{\omega_i}\).  Consider the ring homomorphism defined on generators as follows:
  \begin{align*}
    \phi\colon \ZZII[M] &\to \h^*(\R(\levisubgroup{H}))
  \end{align*}
  \begin{align*}
    e^{\tau_1} & \mapsto e_4^2e_1^{-1} & 
    e^{\tau_2} & \mapsto e_2e_1^{-3} & 
    e^{\tau_3} & \mapsto e_2e_4e_1^{-2} & 
    e^{\tau_4} & \mapsto e_3e_1^{-1}
  \end{align*}
  Note \(\phi\) is well-defined.  Although the definition of \(\phi\) is different from the definition of the homomorphism \(\phi\) in \cref{lemmatatecohfreemonoid} (where we mapped \(e^{\tau_i}\) to \(e_{\tau_i}\)), we can argue as in the proof of \cref{lemmatatecohfreemonoid} that \(\phi\) is an isomorphism.  Indeed, by \cref{prop:Tate-of-H}, \(\phi(e^{\tau_i}) = e_{\tau_i} + \text{ smaller terms}\), and we find that \(\phi\) maps the \(\ZZII\)-basis \(\{e^{\tau} \mid \tau\in M\}\) of \(\ZZII[M]\) to a \(\ZZII\)-basis of \(\h^*(\R(\levisubgroup{H}))\).

  From Table~\ref{tablef4_2}, we see that \([\symsum{}(\omega_i)] = [\symsum{H}(\tau_i)] = e_{\tau_i}\) in $\h^\ast(R(\levisubgroup{H}))$.  Let us identify these elements in terms of the multiplicative generators \(\phi(e^{\tau_i})\).  Note that:
  \begin{compactitem}
  \item \(\symsum{H}(\omega+\nu)=\symsum{H}(\omega)\cdot \symsum{H}(\nu)\) for any \(\nu\in\cellclosureRR{H}\)
  \item \(\symsum{H}(m\cdot \nu) = \symsum{H}(\nu)^m\) for any \(\nu\in\cellclosureRR{H}\) and any \(m\in\ZZ\)
  \item \([\symsum{H}(2\omega)] = [\symsum{H}(\omega)^2]\) in Tate cohomology in view of the fact that Tate cohomology is 2-torsion.
  \end{compactitem}
  Indeed, the first two observations follow as $\Weyl{H}.\nu=\left\{\nu\right\}$, and the third follows as Tate cohomology is \(2\)-torsion.
  Using these observations, we can identify \(e_{\tau_1}\), \(e_{\tau_2}\) and \(e_{\tau_4}\) as follows:
  \begin{align*}
    e_{\tau_1}  &= \phi(e^{\tau_1}) \\
    e_{\tau_2}  &= \phi(e^{\tau_2})  \\
    e_{\tau_3} &= \phi(e^{\tau_3}) + P(e_{\tau_1},e_{\tau_2},e_{\tau_4})\\
    e_{\tau_4} &= \phi(e^{\tau_4})
  \end{align*}
  In the identification of \(e_{\tau_3}\), \(P\) denotes a polynomial in three variables.  We obtain this identification of \(e_{\tau_3}\) from \cref{prop:Tate-of-H}.  (In fact, the only weights  \(\tau\in\fixmonoid\) such that \(\tau \lessdominant{H} \tau_3\) are \(0\), \(\tau_1\) and \(\tau_4\).  So \(P(e_{\tau_1},e_{\tau_2},e_{\tau_4})\) is a linear combination of \(1\), \(e_{\tau_1}\) and \(e_{\tau_4}\).)

  Noting that $e^{\tau_1},e^{\tau_2},e^{\tau_4}$ is a regular sequence in $\ZZII[M]$, we find from \cref{prop:key-Tate-argument} that
  \[
    \h^*\left(\R(\levisubgroup{H}) / (\rsymsum{}(\omega_i) \mid i \in \{1,2,4\})\right) \cong \frac{\ZZII[e^{\tau_3}]}{\left(e^{\tau_3}+\text{rk}(e^{\tau_3})\right)^2}
  \]
  In this quotient ring, \([\rsymsum{}(\omega_3)] = e^{\tau_3} + \text{rk}(e^{\tau_3})\) is a zero divisor.
  Using \cite[Lemma~1.6]{hemmert1}, we conclude that \(\h^*\left(\R(\levisubgroup{H}) / (\rsymsum{}(\omega_i) \mid i \in \{1,2,3,4\})\right)\) is an exterior algebra on one generator, as claimed.  It remains to apply \eqref{eq:main-iso-1}.
\end{proof}

\begin{proof}[Proof of \cref{thm:F4}]
  As noted in \eqref{eq:main-iso-1}, we have an isomorphism between the Tate cohomology computed above and the Witt ring of \(G/\levisubgroup{H}\). So all that remains to show is that the generators of the exterior algebra are in $\Witt^{-1}(F_4/\levisubgroup{H})$.  In each case except the case in \cref{lemmaf4exceptionalcase}, the generators are of the form \([\alphaU(u)]\) for certain elements \(u \in \R(\levisubgroup{H})\) with
  \[
    u + u^* = i^*\rsymsum{}(\omega_i) + \sum_{j\neq i} \mu_j \cdot i^* \rsymsum{}(\omega_j)
  \]
  for certain coefficients \(\mu_j \in \R(\levisubgroup{H})\).  (See last sentence in \cref{prop:key-Tate-argument}.)  As all complex representations of $F_4$ are of real type, each factor \(\rsymsum{}(\omega_j)\) appearing here is of real type.  Replacing our choice of \(u\) if necessary, we may moreover assume that all \(\mu_j\) are of real type.  We can then apply \cref{lem:degrees}.  In the case considered in \cref{lemmaf4exceptionalcase}, we instead find
  \[
    u + u^* = i^*\rsymsum{}(\omega_3)^2 + \sum_{j\neq 3} \mu_j \cdot i^* \rsymsum{}(\omega_j)
  \]
  (see \cite[Lemma~1.6]{hemmert1}).  We can conclude as before.
\end{proof}
\begin{rem}
  For a geometric description of the flag variety corresponding to the diagram in \cref{lemmaf4exceptionalcase}, see \cref{eg:F4-exception}.
  More generally, such geometric descriptions are known in all cases when \(\levisubgroup{H}\) is maximal, i.e.\ when \(\simpleroots{}\setminus H\) consists of a single root \cite{carrgaribaldi}.
\end{rem}

\bibliographystyle{alphaurl}
\bibliography{main}

\end{document}